\documentclass[11pt,a4paper]{amsart}

\title[Explicit Serre duality]{Explicit Serre duality on complex spaces}

\author{Jean Ruppenthal \& H{\aa}kan Samuelsson Kalm \& Elizabeth Wulcan}
\thanks{The first author was supported by the Deutsche Forschungsgemeinschaft (DFG, German Research Foundation),
grant RU 1474/2 within DFG's Emmy Noether Programme.}
\thanks{The second and the third author were partially supported by the Swedish Research Council.}
\address{H. Samuelsson Kalm, E. Wulcan, Department of Mathematical Sciences, Division of Mathematics, University of Gothenburg and 
Chalmers University of Technology, SE-412 96 G\"{o}teborg, Sweden}
\email{hasam@chalmers.se, wulcan@chalmers.se}
\address{J. Ruppenthal, Department of Mathematics, University of Wuppertal, Gaussstr.\ 20, 42119 Wuppertal, Germany}
\email{ruppenthal@uni-wuppertal.de}
\date{\today}

\usepackage[english]{babel}
\usepackage{amsmath,calligra}
\usepackage{amsthm,amssymb,latexsym}
\usepackage{url,ae}
\usepackage{t1enc}
\usepackage{mathrsfs}
\usepackage{graphicx}
\usepackage{eufrak} 
\usepackage{geometry}
\newtheorem{proposition}{Proposition}[section]
\newtheorem{theorem}[proposition]{Theorem}
\newtheorem{lemma}[proposition]{Lemma}
\newtheorem{corollary}[proposition]{Corollary}

\theoremstyle{definition}
\newtheorem{definition}[proposition]{Definition}
\newtheorem{example}[proposition]{Example}
\newtheorem{remark}[proposition]{Remark}

\numberwithin{equation}{section}

\DeclareMathOperator{\Ext}{\mathscr{E}\text{\kern -3pt {\calligra\Large xt}}\,\,}
\newcommand{\C}{\mathbb{C}}
\newcommand{\debar}{\bar{\partial}}
\newcommand{\dbar}{\bar{\partial}}
\newcommand{\A}{\mathscr{A}}
\newcommand{\B}{\mathscr{B}}
\newcommand{\F}{\mathscr{F}}
\newcommand{\G}{\mathscr{G}}
\newcommand{\HH}{\mathscr{H}}
\newcommand{\R}{\mathbb{R}}
\newcommand{\J}{\mathcal{J}}

\newcommand{\W}{\mathcal{W}}
\newcommand{\PM}{\mathcal{PM}}

\newcommand{\hol}{\mathscr{O}}

\newcommand{\V}{\mathcal{V}}

\newcommand{\K}{\mathscr{K}}

\newcommand{\Proj}{\mathscr{P}}

\def\newop#1{\expandafter\def\csname #1\endcsname{\mathop{\rm #1}\nolimits}}
\newop{span}

\begin{document}
\nocite{*}
\bibliographystyle{plain}

\begin{abstract}
In this paper we use recently developed calculus
of residue currents
together with integral formulas to give a new explicit analytic realization, as well as a new analytic proof, of 
Serre duality on any reduced pure $n$-dimensional paracompact complex space $X$. 
At the core of the paper is the introduction of certain fine sheaves $\B_X^{n,q}$ of 
currents on $X$ of bidegree $(n,q)$, such that the Dolbeault complex $(\B_X^{n,\bullet},\,\debar)$ 
becomes, in a certain sense, a dualizing complex. 
In particular, if $X$ is Cohen-Macaulay 
then $(\B_X^{n,\bullet},\,\debar)$ is an explicit fine resolution of the Grothendieck dualizing sheaf.
\end{abstract}

\maketitle
\thispagestyle{empty}

\section{Introduction}
Let $X$ be a complex $n$-dimensional manifold and let $F\to X$ be a complex vector bundle. 
Let $\mathcal{E}^{0,q}(X,F)$ denote the space of smooth $F$-valued $(0,q)$-forms on $X$
and let $\mathcal{E}^{n,q}_{c}(X,F^*)$ denote the space of smooth compactly supported $(n,q)$-forms on $X$
with values in the dual vector bundle $F^*$. 
Serre duality, \cite{Serre}, can be formulated analytically as follows: \emph{There is a non-degenerate pairing}
\begin{equation}\label{analSerre}
H^q \left( \mathcal{E}^{0,\bullet}(X,F), \debar \right)\times H^{n-q} \left(\mathcal{E}^{n,\bullet}_{c}(X,F^*), \debar \right) \to \C,
\end{equation}
\begin{equation*}
([\varphi]_{\debar}, [\psi]_{\debar}) \mapsto \int_X \varphi\wedge\psi,
\end{equation*}
\emph{provided that $H^q(\mathcal{E}^{0,\bullet}(X,F),\debar)$ and $H^{q+1}(\mathcal{E}^{0,\bullet}(X,F),\debar)$ 
are Hausdorff considered as topological vector spaces.}
If we set $\F:=\hol(F)$ and $\F^*:=\hol(F^*)$ and let $\mathit{\Omega}_X^n$ denote the sheaf of holomorphic $n$-forms on $X$,
then one can, via the Dolbeault isomorphism, rephrase
Serre duality more algebraically: There is a non-degenerate pairing
\begin{equation}\label{eq:algSerre}
H^q(X, \F)\times H^{n-q}_{c}(X, \F^*\otimes \mathit{\Omega}_X^{n}) \to \C,
\end{equation} 
realized by the cup product, provided that $H^q(X,\F)$ and $H^{q+1}(X,\F)$ are Hausdorff. 
In this formulation
Serre duality has been generalized to complex spaces, see, e.g., Hartshorne
\cite{Hartshorne1}, \cite{Hartshorne2}, and Conrad \cite{Conrad} for the algebraic setting and Ramis-Ruget \cite{RaRu}
and Andreotti-Kas \cite{AK} for 
the analytic. 
In fact, if $X$ is a pure $n$-dimensional paracompact complex space 
that in addition is Cohen-Macaulay, then again there is a perfect pairing
\eqref{eq:algSerre} if we construe $\mathit{\Omega}^n_X$ as the {\em Grothendieck dualizing sheaf}
that we will get back to shortly.
If $X$ is not Cohen-Macaulay things get more involved and $H^{n-q}_{c}(X,\F^*\otimes\mathit{\Omega}^n_X)$
is replaced by $\textrm{Ext}^{-q}_{c}(X;\F, \mathbf{K}^{\bullet})$, where $\mathbf{K}^{\bullet}$
is the {\em dualizing complex} in the sense of Ramis-Ruget
\cite{RaRu}, that is a certain complex of $\hol_X$-modules
with coherent cohomology.

To our knowledge there is no such explicit analytic realization 
of Serre duality as \eqref{analSerre} in the case of singular spaces. 
In fact, {\em verbatim} the pairing
\eqref{analSerre} cannot realize Serre duality in general since the Dolbeault complex 
$(\mathcal{E}_X^{0,\bullet},\debar)$\footnote{See below for the definition of $\mathcal{E}_X^{p,q}$; the sheaf of smooth
$(p,q)$-forms on $X$.} in general does not provide a resolution of $\hol_X$.
In this paper we replace the sheaves of smooth forms by certain 
 fine sheaves of currents $\A_X^{0,q}$ and $\B_X^{n,n-q}$ 
that are smooth on $X_{reg}$ and such that \eqref{analSerre} with
$\mathcal{E}^{0,\bullet}$ and $\mathcal E^{n,\bullet}$ replaced by
$\A^{0,\bullet}$ and $\B^{n,\bullet}$, respectively, indeed realizes Serre duality.

We will say that a complex $(\mathscr{D}^{\bullet}_X,\delta)$ of fine sheaves is a 
\emph{dualizing Dolbeault complex} for a coherent sheaf $\F$ if $(\mathscr{D}^{\bullet}_X,\delta)$ has coherent cohomology
and if there is a non-degenerate pairing $H^q(X,\F)\times H^{n-q}(\mathscr{D}^{\bullet}_{c}(X),\delta)\to \C$.
The relation to the Ramis-Ruget dualizing complex is not completely
clear to us, but we still find this terminology convenient.
For instance, $(\B^{n,\bullet}_X,\debar)$ is a dualizing Dolbeault complex for $\hol_X$.

At this point it is appropriate to mention that Ruget in \cite{Ruget}
shows, using Coleff-Herrera residue theory, that there is an injective morphism 
$\mathbf{K}^{\bullet}_X\to \mathscr{C}_X^{n,\bullet}$, where $\mathscr{C}_X^{n,\bullet}$ is the sheaf of germs of currents
on $X$ of bidegree $(n,\bullet)$.

\begin{center}
---
\end{center}

Let $X$ be a reduced complex space of pure dimension $n$.
Recall that every point in $X$ has a neighborhood $V$ that can be embedded into 
some pseudoconvex domain $D\subset\C^N$, $i\colon V\to D$, and that
$\hol_V\cong \hol_{D}/\J_V$, where $\J_V$ is the radical ideal sheaf in $D$ defining
$i(V)$. Similarily, a $(p,q)$-form $\varphi$ on $V_{reg}$ is said to be smooth on $V$ if there is a 
smooth $(p,q)$-form $\tilde{\varphi}$ in $D$ such that $\varphi=i^*\tilde{\varphi}$ on $V_{reg}$.
It is well known that the so defined smooth forms on $V$ define an intrinsic sheaf $\mathcal{E}_X^{p,q}$ on $X$.
The currents of bidegree $(p,q)$ on $X$ are defined as the dual of the space of compactly
supported smooth $(n-p,n-q)$-forms on $X$. More concretely, given
a local embedding $i\colon V\to D$, for any $(p,q)$-current $\mu$ on $V$, $\tilde{\mu}:=i_*\mu$ is a 
current of bidegree $(p+N-n, q+N-n)$ in $D$ with the property that $\tilde{\mu}.\xi=0$ for every test form
$\xi$ in $D$ such that $i^*\xi|_{V_{reg}}=0$. Conversely, if $\tilde{\mu}$ is a current in $D$ with this property,
then it defines a current on $V$ (with a shift in bidegrees). We will often suggestively write
$\int \mu\wedge \xi$ for the action of the current $\mu$ on the test form $\xi$.

A current $\mu$ on $X$ is said to have the \emph{standard extension property} (SEP) with respect to
a subvariety $Z\subset X$ if for all open $\mathcal U\subset X$, 
$\chi(|h|/\epsilon)\mu|_{\mathcal U} \to \mu|_{\mathcal U}$ as
$\epsilon \to 0$, where  $\mu|_{\mathcal U}$ denotes the restriction
of $\mu$ to $\mathcal U$, $\chi$ is a smooth regularization of 
the characteristic function of $[1,\infty)\subset \R$, and $h$ is any
holomorphic tuple that does not vanish identically on any irreducible
component of $Z\cap \mathcal U$. 
If $Z=X$ we simply say that $\mu$ has the SEP on $X$.
In particular, two currents with the SEP on $X$ are equal on $X$ if and only if they are equal on $X_{reg}$.

We will say that a current $\mu$ on $X$ has 
{\em principal value-type singularities} if $\mu$ is locally integrable outside a hypersurface and has the SEP
on $X$. Notice that
if $\mu$ has principal value-type singularities and $h$ is a generically non-vanishing holomorphic tuple
such that $\mu$ is locally integrable outside $\{h=0\}$, then the action of $\mu$ on a test form $\xi$ can be computed as
\begin{equation*}
\lim_{\epsilon\to 0} \int_X \chi(|h|/\epsilon)\mu\wedge \xi,
\end{equation*} 
where the integral now is an honest integral of an integrable form on the manifold $X_{reg}$.

\medskip

By using integral formulas and residue theory, Andersson and the second author introduced in \cite{AS}
fine sheaves $\A_X^{0,q}$ (i.e., modules over $\mathcal{E}_X^{0,0}$) of $(0,q)$-currents 
with the SEP on $X$, containing $\mathcal{E}_X^{0,q}$, and coinciding with $\mathcal{E}_{X_{reg}}^{0,q}$ on $X_{reg}$, 
such that the associated Dolbeault complex yields a resolution of
$\hol_X$, see \cite[Theorem~1.2]{AS}. Notice that it follows that $H^q(\A^{0,\bullet}(X),\debar)\simeq H^q(X,\hol_X)$.
Moreover, by a standard construction it then follows that each cohomology class in $H^q(\A^{0,\bullet}(X),\debar)$ has a smooth 
representative; cf.\ Section~\ref{sec:cup} below.
Similar to the construction of the $\A$-sheaves in \cite{AS} 
we introduce our sheaves $\B_X^{n,q}$ of $(n,q)$-currents and show that these currents have the SEP on
$X$, that $\mathcal{E}_X^{n,q}\subset \B_X^{n,q}$, and that $\B_X^{n,q}$ coincides with
$\mathcal{E}_{X}^{n,q}$ on $X_{reg}$; cf.\
Proposition~\ref{prop:A}. Moreover, by Theorem~\ref{thm:debar}, $\debar \colon \B_X^{n,q} \to \B_X^{n,q+1}$,
where of course $\debar$ is defined by duality: $\int \debar\mu\wedge \xi:=\pm\int\mu\wedge \debar\xi$
for currents $\mu$ and test forms $\xi$ on $X$.
By adapting the constructions in \cite{AS} to the setting of $(n,q)$-forms we get the following
semi-global homotopy formula for $\debar$.

\begin{theorem}\label{thm:Homotopy}
Let $V$ be a pure $n$-dimensional analytic subset of a pseudoconvex domain $D\subset\C^N$, let 
$D'\Subset D$, and put $V'=V\cap D'$.
There are integral operators
\begin{equation*}
\check{\mathscr{K}}\colon \B^{n,q}(V)\to \B^{n,q-1}(V'), \quad
\check{\mathscr{P}}\colon \B^{n,q}(V)\to \B^{n,q}(V'),
\end{equation*}
such that if $\psi\in \B^{n,q}(V)$, then the homotopy formula 
\begin{equation*}
\psi = \debar \check{\mathscr{K}} \psi + \check{\mathscr{K}}(\debar \psi) + \check{\mathscr{P}} \psi
\end{equation*}
holds on $V'$.
\end{theorem}

The integral operators $\check{\K}$ and $\check{\Proj}$ are given by kernels $k(z,\zeta)$ and $p(z,\zeta)$ that
are respectively integrable and smooth on $Reg(V_z)\times Reg(V'_{\zeta})$  
and that have principal value-type singularities at the singular locus of $V\times V'$. In particular, 
one can compute $\check{\K}\psi$ and $\check{\Proj}\psi$ as
\begin{equation*}
\check{\K}\psi(\zeta)=\lim_{\epsilon\to 0} \int_{V_z} \chi(|h(z)|/\epsilon)k(z,\zeta)\wedge \psi(z), \quad
\check{\Proj}\psi(\zeta)=\lim_{\epsilon\to 0} \int_{V_z} \chi(|h(z)|/\epsilon)p(z,\zeta)\wedge \psi(z),
\end{equation*}
where $\chi$ is as above, 
$h$ is a holomorphic tuple such that $\{h=0\}=V_{sing}$, and where the limit is understood in the sense of currents.
We use our integral operators to prove the following result.

\begin{theorem}\label{thm:resol}
Let $X$ be a reduced complex space of pure dimension $n$. 
The cohomology sheaves $\omega_X^{n,q}:=\mathscr{H}^q(\B^{n,\bullet}_X, \debar)$ of the sheaf complex
\begin{equation}\label{RSWkomplex}
0\to \B_X^{n,0} \stackrel{\debar}{\longrightarrow} \B_X^{n,1} \stackrel{\debar}{\longrightarrow}
\cdots \stackrel{\debar}{\longrightarrow} \B_X^{n,n} \to 0
\end{equation}
are coherent. If $X$ is Cohen-Macaulay, then
\begin{equation}\label{RSWkomplexCM}
0\to \omega_X^{n,0}\hookrightarrow \B_X^{n,0} \stackrel{\debar}{\longrightarrow} \B_X^{n,1} \stackrel{\debar}{\longrightarrow}
\cdots \stackrel{\debar}{\longrightarrow} \B_X^{n,n} \to 0
\end{equation}
is exact.
\end{theorem}

In fact, our proof of Theorem~\ref{thm:resol} shows that if $V\subset X$ is identified with an analytic codimension $\kappa$ 
subset of a pseudoconvex domain $D\subset \C^N$, then $\omega_V^{n,q}\cong \Ext^{\kappa+q}(\hol_D/\J_V, \mathit{\Omega}^N_D)$,
where $\mathit{\Omega}^N_D$ is the canonical sheaf on $D$. Hence, we get a concrete analytic realization of these 
$\Ext$-sheaves. 

The sheaf $\omega_V^{n,0}$ of $\debar$-closed currents in $\B_V^{n,0}$
is in fact equal to the sheaf of $\debar$-closed meromorphic currents on $V$
in the sense of Henkin-Passare \cite[Definition~2]{HePa}, cf.\ \cite[Example~2.8]{AS}. 
This sheaf was introduced earlier by Barlet in a different way
in \cite{Barlet}; cf.\ also \cite[Remark~5]{HePa}.
In case $X$ is Cohen-Macaulay
$\Ext^{\kappa}(\hol_D/\J_V, \mathit{\Omega}^N_D)$ is by definition the Grothendieck dualizing sheaf. Thus, \eqref{RSWkomplexCM}
can be viewed as a concrete analytic fine resolution of the Grothendieck dualizing sheaf in
the Cohen-Macaulay case.

\begin{center}
---
\end{center}

Let $\varphi$ and $\psi$ be sections of $\A_X^{0,q}$ and $\B_X^{n,q'}$ respectively. Since $\varphi$ and $\psi$
then are smooth on the regular part of $X$, the exterior product $\varphi|_{X_{reg}}\wedge \psi|_{X_{reg}}$ 
is a smooth $(n,q+q')$-form on $X_{reg}$. In Theorem~\ref{thm:parning} we show that 
$\varphi|_{X_{reg}}\wedge \psi|_{X_{reg}}$ has a natural extension across $X_{sing}$ as a current 
with principal value-type singularities; we denote this current by $\varphi\wedge\psi$. 
Moreover, it turns out that the Leibniz rule 
$\debar(\varphi\wedge \psi)= \debar\varphi\wedge \psi + (-1)^{q}\varphi\wedge \debar\psi$ holds.
Now, if $q'=n-q$ and $\psi$ (or $\varphi$) has compact support, then
$\int \varphi\wedge \psi$ (i.e., the action of $\varphi\wedge \psi$ on $1$)
gives us a complex number. Since the Leibniz
rule holds we thus get a pairing, a \emph{trace map}, on cohomology level:
\begin{equation*}
Tr \colon H^q\left( \A^{0,\bullet}(X), \debar \right) \times H^{n-q}\left( \B^{n,\bullet}_{c}(X), \debar \right) \to \C, 
\end{equation*}
\begin{equation*}
Tr([\varphi]_{\debar},[\psi]_{\debar}) = \int_X \varphi\wedge \psi,
\end{equation*}
where $\A^{0,q}(X)$ denotes the global sections of $\A_X^{0,q}$ and
$\B^{n,q}_{c}(X)$ denotes the global sections of $\B_X^{n,q}$ with compact support. 
It causes no problems to insert a locally free sheaf:
If $F\to X$ is a vector bundle, $\F=\hol(F)$ the associated locally free sheaf, and $\F^*=\hol(F^*)$ the dual sheaf,
then the trace map gives a pairing $\F\otimes\A^{0,q}(X)\times \F^*\otimes \B_{c}^{n,n-q}(X) \to \C$.

\begin{theorem}\label{thm:main2}
Let $X$ be a paracompact reduced complex space of pure dimension $n$ and $\F$ a locally free sheaf on $X$. 
If $H^q(X,\F)$ and $H^{q+1}(X,\F)$, considered as topological vector
spaces, are Hausdorff (e.g., finite dimensional), 
then the pairing
\begin{equation*}
H^q\left(\F\otimes \A^{0,\bullet}(X), \debar\right) \times H^{n-q}\left(\F^*\otimes \B_{c}^{n,\bullet}(X), \debar\right) \to \C, 
\quad ([\varphi],[\psi]) \mapsto \int_X \varphi \wedge \psi
\end{equation*}
is non-degenerate.
\end{theorem}

Since the $\A$-cohomology has smooth representatives,
it follows that if $X$ is compact and $\psi$ is a smooth $\debar$-closed $(n,q)$-form on $X$,
then there is a $u\in \B^{n,q-1}(X)$ (in particular $u$ is smooth on $X_{reg}$) such that $\debar u=\psi$ 
if and only if $\int_X\varphi\wedge\psi=0$ for all smooth $\debar$-closed $(0,n-q)$-forms $\varphi$.

Notice also that, by \cite[Theorem~1.2]{AS}, the complex $(\F\otimes \A^{0,\bullet}_X, \debar)$ is a fine resolution of $\F$ 
and so, via the Dolbeault isomorphism, Theorem~\ref{thm:main2} gives us a non-degenerate pairing
\begin{equation*}
H^q(X,\F) \times H^{n-q}(\F^*\otimes \B_{c}^{n,\bullet}(X), \debar) \to \C.
\end{equation*}
The complex $(\F^*\otimes \B^{n,\bullet}_X,\debar)$ is thus a concrete analytic
dualizing Dolbeault complex for $\F$.
If $X$ is Cohen-Macaulay, then $(\F^*\otimes \B^{n,\bullet}_X,\debar)$
is, by Theorem~\ref{thm:resol} above, a fine resolution of the sheaf $\F^*\otimes \omega^{n,0}_X$ and so
Theorem~\ref{thm:main2} yields in this case a non-degenerate pairing
\begin{equation*}
H^q(X, \F) \times H^{n-q}_{c}(X, \F^*\otimes \omega^{n,0}_X) \to \C.
\end{equation*}
In Section~\ref{sec:cup} we show that this pairing also can be realized as the cup product in \v{C}ech cohomology.

\begin{remark}
By \cite[Th\'eor\`{e}me 2]{RaRu} there is another non-degenerate pairing 
\begin{equation*}
H^q_{c}(X,\F) \times \textrm{Ext}^{-q}(X; \F, \mathbf{K}_X^{\bullet}) \to \C
\end{equation*}
if $H^q_{c}(X,\F)$ and $H^{q+1}_{c}(X,\F)$ are Hausdorff. 
In view of this we believe that one can show that, under the same assumption, the pairing
\begin{equation*}
H^q\left(\F\otimes \A^{0,\bullet}_{c}(X), \debar\right) \times H^{n-q}\left(\F^*\otimes \B^{n,\bullet}(X), \debar\right) \to \C, 
\quad ([\varphi],[\psi]) \mapsto \int_X \varphi \wedge \psi
\end{equation*}
is non-degenerate but we do not pursue this question in this paper.
\end{remark}

\medskip

{\bf Acknowledgment:} We would like to thank Mats Andersson for valuable discussions and comments
that have simplified some proofs significantly. We would also like to thank the referee
for many important comments.

\section{Preliminaries}\label{sec:prelim}
Our considerations here are local or semi-global so let $V$ be a pure $n$-dimensional analytic subset of a pseudoconvex domain 
$D\subset \C^N$. Throughout we let $\kappa = N-n$ denote the
codimension of $V$.

\subsection{Pseudomeromorphic currents on a complex space}\label{ssec:PM}
In $\C_z$ the principal value current $1/z^m$ can be defined, e.g., as the limit as $\epsilon \to 0$ in the sense of currents
of $\chi(|h(z)|/\epsilon)/z^m$, where $\chi$ is a smooth regularization of the characteristic function of $[1,\infty)\subset \R$
and $h$ is a holomorphic function vanishing at $z=0$, or as the value at $\lambda=0$ of the analytic continuation
of the current-valued function $\lambda \mapsto |h(z)|^{{2\lambda}}/z^m$. Regularizations of the form $\chi(|h|/\epsilon)\mu$
of a current $\mu$ occur frequently in this paper and throughout $\chi$ will
denote a smooth regularization of the characteristic function of $[1,\infty)\subset \R$.
The \emph{residue current} $\debar(1/z^m)$ can be computed as the 
limit of $\debar\chi(|h(z)|/\epsilon)/z^m$ or as the value at $\lambda=0$ of $\lambda \mapsto \debar |h(z)|^{{2\lambda}}/z^m$.
Since tensor products of currents are well-defined we can form the current
\begin{equation}\label{elementary}
\tau=\debar \frac{1}{z_1^{m_1}}\wedge \cdots \wedge \debar \frac{1}{z_r^{m_r}}\wedge 
\frac{\gamma(z)}{z_{r+1}^{m_{r+1}}\cdots z_n^{m_n}}
\end{equation}
in $\C^n_z$, where $m_1,\ldots,m_r$ are positive integers, $m_{r+1},\ldots,m_{n}$ are nonnegative integers, and 
$\gamma$ is a smooth compactly supported form. Notice that $\tau$ is anti-commuting in the residue factors 
$\debar(1/z_j^{m_j})$ and commuting in the principal value factors $1/z_k^{m_k}$.
A current of the form \eqref{elementary} is called an \emph{elementary pseudomeromorphic current} 
and we say that a current $\mu$ on $V$ is {\em pseudomeromorphic}, $\mu\in \PM(V)$,
if it is a locally finite sum of pushforwards $\pi_*\tau=\pi_*^1 \cdots \pi_*^{\ell} \tau$ under maps
\begin{equation*}
V^{\ell} \stackrel{\pi^{\ell}}{\longrightarrow} \cdots \stackrel{\pi^{2}}{\longrightarrow} V^1
\stackrel{\pi^{1}}{\longrightarrow} V^0=V,
\end{equation*}
where each $\pi^j$ is either a modification, a simple projection $V^j=V^{j-1}\times Z\to V^{j-1}$, or 
an open inclusion, and $\tau$ is an elementary pseudomeromorphic current on $V^{\ell}$. The sheaf of pseudomeromorphic currents
on $V$ is denoted $\PM_V$.
Since the class of elementary currents is closed under $\dbar$ and $\dbar$ commutes with push-forwards it follows that $\PM_V$ is
closed under $\dbar$. 
Pseudomeromorphic currents were originally introduced in \cite{AW} but with a more restrictive definition;
simple projections were not allowed. In this paper we adopt the definition of pseudomeromorphic
currents in \cite{AS}.

\begin{example}
Let $f\in \hol(V)$ be generically non-vanishing and let $\alpha$ be a smooth form on $V$. Then
$\alpha/f$ is a semi-meromorphic form on $V$ and it defines a \emph{semi-meromorphic current},
also denoted $\alpha/f$, on $V$ by
\begin{equation}\label{HLmap}
\xi \mapsto \lim_{\epsilon \to 0}\int_V \chi(|h|/\epsilon) \frac{\alpha}{f}\wedge \xi, 
\end{equation} 
where $\xi$ is a test form on $V$ and $h\in \hol(V)$ is generically non-vanishing and vanishes on $\{f=0\}$.
That \eqref{HLmap} indeed gives a well-defined current is proved in \cite{HL}; the 
existence of the limit in \eqref{HLmap} relies on Hironaka's theorem on resolution of singularities.
Let $\pi \colon \tilde{V}\to V$ be a smooth modification such that $\{\pi^*f=0\}$
is a normal crossings divisor. Locally on $\tilde{V}$ one can thus choose coordinates so that
$\pi^*f$ is a monomial. One can then show that the semi-meromorphic current $\alpha/f$ is the push-forward
under $\pi$ of elementary pseudomeromorphic currents \eqref{elementary} with $r=0$;
hence, $\alpha/f\in \PM(V)$.

The $(0,1)$-current $\debar(1/f)\in \PM(V)$ is the residue current of
$f$. 
Since the action of $1/f$ on test forms is given by
\eqref{HLmap} with $\alpha=1$ it follows from Stokes' theorem that 
\begin{equation*}
\debar \frac{1}{f}.\, \xi = \lim_{\epsilon \to 0}\int_V  \frac{\debar\chi(|h|/\epsilon)}{f}\wedge \xi. 
\end{equation*}
\hfill $\qed$
\end{example}

One crucial property of pseudomeromorphic currents is the following, see, e.g., \cite[Proposition~2.3]{AS}.

\medskip

\noindent {\bf Dimension principle.}
\emph{Let $\mu\in \PM(V)$ and assume that $\mu$ has support on the subvariety $Z\subset V$. 
If $\textrm{dim}\, V - \textrm{dim}\, Z > q$ and $\mu$ has bidegree $(*,q)$, then $\mu=0$.}

\medskip

Pseudomeromorphic currents can be ``restricted'' to analytic subsets. In fact, following \cite{AW}, 
if $\mu\in\PM(V)$ and $Z\subset V$
is an analytic subset, then $\mu|_{V\setminus Z}$ has a natural pseudomeromorphic extension 
to $V$ denoted $\mathbf{1}_{V\setminus Z}\mu$. Thus, $\mathbf{1}_Z\mu:=\mu-\mathbf{1}_{V\setminus Z}\mu$ is a 
pseudomeromorphic current on $V$ with support on $Z$. In \cite{AW}, $\mathbf{1}_{V\setminus Z}\mu$ is defined
as $|h|^{2\lambda}\mu|_{\lambda=0}$, where $h$ is a holomorphic tuple such that $\{h=0\}=Z$, but it can also
be defined as $\lim_{\epsilon\to 0}\chi(|h|/\epsilon)\mu$; cf.\ \cite{AW3} and \cite[Lemma~6]{LS}.  
It follows that if $\mu=\pi_*\tau$, then $\mathbf{1}_Z\mu=\pi_*(\mathbf{1}_{\pi^{-1}(Z)}\tau)$.
Notice that a pseudomeromorphic current $\mu$ has the SEP if and only if $\mathbf{1}_Z \mu=0$ for all germs of 
analytic sets $Z$ with positive codimension.
We will denote by $\W_V$ the subsheaf of $\PM_V$ of currents with
the SEP. 
It is closed under multiplication by smooth
forms and if $\pi\colon \tilde{V}\to V$ is either a modification or a simple projection then $\pi_*\colon \W(\tilde{V})\to \W(V)$. 

A natural subclass of $\W(V)$ is the class of {\em almost semi-meromorphic currents} on $V$; a current
$\mu$ on $V$ is said to be almost semi-meromorphic if there is a smooth modification $\pi\colon \tilde{V}\to V$
and a semi-meromorphic current $\tilde{\mu}$ on $\tilde{V}$ such that
$\pi_* \tilde{\mu} = \mu$, see \cite{AS}. 
Notice that almost semi-meromorphic currents are generically smooth and have principal value-type singularities.
Let $\mu$ be an almost semi-meromorphic current. Following \cite{AW3},
we let $ZSS(\mu)$ (the Zariski-singular support of $\mu$) be the smallest Zariski-closed set
outside of which $\mu$ is smooth. 
The following result can be found in \cite{AW3}; the last part is \cite[Proposition~2.7]{AS}.
\begin{proposition}\label{multprop}
Let $a$ be an almost semi-meromorphic current on $V$ and let $\mu\in\PM(V)$. 
Then there is a unique pseudomeromorphic current $a\wedge \mu$ on $V$
that coincides with $a\wedge\mu$ outside of $ZSS(a)$ and such that
$\mathbf{1}_{ZSS(a)}a\wedge\mu=0$. 
If $\mu\in\W(V)$, then $a\wedge\mu\in\W(V)$.
\end{proposition}

If $\mu\in \PM(V_z)$ and $\nu\in\PM(W_{\zeta})$ then we will denote the current $(\mu\otimes 1)\wedge (1\otimes \nu)$
on $V_z\times W_{\zeta}$ by $\mu(z)\wedge\nu(\zeta)$, or sometimes $\mu\wedge\nu$ if there is no risk of confusion,
and refer to it as the \emph{tensor product} of $\mu$ and $\nu$. From \cite{AW3} we have that
$\mu(z)\wedge\nu(\zeta)\in \PM(V\times W)$ and that $\mu(z)\wedge\nu(\zeta)\in \W(V\times W)$ if
$\mu\in\W(V)$ and $\nu\in\W(W)$.

We will also have use for the following slight variation of \cite[Theorem~1.1 (ii)]{Anot}.

\begin{proposition}\label{matssats}
Let $Z\subset V$ be a pure dimensional analytic subset and let $\J_Z\subset \hol_V$
be the ideal sheaf of holomorphic functions vanishing on $Z$. 
Assume that $\tau\in \PM(V)$ has the SEP with respect to $Z$ and that $h \tau=dh\wedge \tau=0$ for all $h\in \J_Z$.
Then there is a current $\mu\in \PM(Z)$ with the SEP such that $\iota_* \mu = \tau$,
where $\iota\colon Z\hookrightarrow V$ is the inclusion.
\end{proposition}

\begin{proof}
Let $i\colon V\hookrightarrow D$ be the inclusion. By \cite[Theorem~1.1 (i)]{Anot} we have that 
$i_*\tau \in \PM(D)$. It is straightforward to verify that $i_*\tau$ has the SEP with respect to 
$Z$ considered now as a subset of $D$ and that $h i_*\tau=dh\wedge i_*\tau=0$ for all $h\in \J_Z$, where we now consider $\J_Z$
as the ideal sheaf of $Z$ in $D$. Hence, it is sufficient to show the proposition when $V$ is smooth.
To this end, we will see that there is a current $\mu$ on $Z$ such that $\iota_*\mu=\tau$; then 
the proposition follows from \cite[Theorem~1.1 (ii)]{Anot}.

The existence of such a $\mu$ is equivalent to that $\tau.\xi=0$ for all test forms $\xi$
such that $\iota^*\xi=0$ on $Z_{reg}$. By, e.g., \cite[Proposition~2.3]{AS} and the assumption on $\tau$
it follows that $\bar{h}\tau=d\bar{h}\wedge \tau=h\tau=dh\wedge \tau=0$ for every $h\in \J_Z$.
Using this it is straightforward to check that if $x\in Z_{reg}$ and $\xi$ is a smooth form such that 
$\iota^*\xi=0$ in a neighborhood of $x$, then $\xi\wedge \tau=0$ in a neighborhood of $x$. 
Thus, if $g$ is a holomorphic tuple in $V$ such that $\{g=0\}=Z_{sing}$, 
then $\chi(|g|/\epsilon)\tau.\xi=0$ for any test form $\xi$ such that $\iota^*\xi=0$ on $Z_{reg}$.
Since $\tau$ has the SEP with respect to $Z$ it follows that $\tau.\xi=0$ for all test forms $\xi$
such that $\iota^*\xi=0$ on $Z_{reg}$. 
\end{proof}

\medskip

\subsection{Residue currents}\label{ssec:rescurr}
We briefly recall the the construction in \cite{AW1} of a residue current 
associated to a generically exact complex of Hermitian vector bundles.

Let $\J_V$ be the radical ideal sheaf in $D$ 
associated with $V\subset D$. Possibly after shrinking $D$ somewhat there is a free resolution
\begin{equation}\label{eq:resolOJ}
0 \to \hol(E_m) \stackrel{f_m}{\longrightarrow} \cdots \stackrel{f_2}{\longrightarrow}
\hol(E_1) \stackrel{f_1}{\longrightarrow} \hol(E_0)
\end{equation}
of $\hol_{D}/\J_V$, where $E_k$ are trivial vector bundles, $E_0$ is the trivial line bundle, $f_k$ are holomorphic mappings,
and $m\leq N$.
The resolution \eqref{eq:resolOJ} induces a complex of vector bundles
\begin{equation*}
0\to E_m \stackrel{f_m}{\longrightarrow} \cdots \stackrel{f_2}{\longrightarrow} E_1
\stackrel{f_1}{\longrightarrow} E_0
\end{equation*}
that is pointwise exact outside $V$. 
For $r\geq 1$, let 
$V^r$ be the set where $f_{\kappa+r}\colon E_{\kappa+r}\to E_{\kappa+r-1}$ does not have optimal 
rank\footnote{For $j\leq \kappa$, the set where $f_j$ does not have optimal rank is $V$.}, and let $V^0:=V_{sing}$.
Then
\begin{equation}\label{ballong}
\cdots \subset V^{k+1} \subset V^k \subset \cdots \subset V^1 \subset V^0 \subset V.
\end{equation} 
By the uniqueness of minimal free resolutions, these sets are in fact independent of the choice of resolution 
\eqref{eq:resolOJ} of $\hol_V=\hol_D/\J_V$, i.e., they are invariants of that sheaf,
and they somehow measure the singularities of $V$. 
Since $V$ has pure dimension it follows from \cite[Corollary~20.14]{Eis} that
\begin{equation*}
\textrm{dim}\, V^r < n-r, \quad r\geq 0.
\end{equation*}
Hence, $V^n=\emptyset$ and so $f_N$ has optimal rank everywhere; we
may thus assume that $m\leq N-1$ in \eqref{eq:resolOJ}.
Recall that $V$ is Cohen-Macaulay if and only if there a
resolution \eqref{eq:resolOJ} with $m=\kappa$ of $\hol_V$, see, e.g., \cite[Chapter~18]{Eis}. Notice that
$V^r=\emptyset$ for $r\geq 1$ if and only if $V$ is Cohen-Macaulay.

\smallskip

Assuming $V$ has positive codimension, given Hermitian metrics on the $E_j$, following \cite{AW1}, one can construct a 
smooth form $u=\sum_{k\geq 1} u_k$ in $D\setminus V$, where $u_k$ is a $(0,k-1)$-form taking values in $E_k$,
such that 
\begin{equation}\label{nablau}
f_1u_1=1, \quad f_{k+1}u_{k+1}=\debar u_k,\, k=1,\ldots,m-1, \quad \debar u_m=0 \quad \textrm{in} \,\, D\setminus V.
\end{equation}
The form $u$ has an extension as an almost semimeromorphic current 
\begin{equation}\label{gulU}
\lim_{\epsilon\to 0} \chi(|F|/\epsilon) u=: U=\sum_{k\geq 1} U_k, 
\end{equation}
where $F$ is a holomorphic tuple in $D$ vanishing on $V$ and $U_k$ is
a $(0,k-1)$-current taking values in $E_k$; one should think of $U$ as
a generalization of the meromorphic current $1/f$ in $D$ when 
$V=f^{-1}(0)$ is a hypersurface.\footnote{In \cite{AW} $U$ was originally defined as the analytic
continuation to $\lambda=0$ of $|F|^{2\lambda}u$. However, in view of \cite[Section~4]{AW3} 
this definition coincides with \eqref{gulU}, see also \cite[Lemma~6]{LS}.}
 The residue current $R=\sum_{k} R_k$ 
associated with $V$ is then
defined by
\begin{equation}\label{lajbans}
R_k=\debar U_k - f_{k+1}U_{k+1},\, k=1,\ldots,m-1, \quad R_m=\debar U_m.
\end{equation}
Hence, $R_k$ is a pseudomeromorphic $(0,k)$-current in $D$ with values in $E_k$, and from \eqref{nablau}
it follows that $R$ has support on $V$. By the dimension principle, thus $R=R_\kappa+\cdots+R_m$. 
Notice that if $V$ is Cohen-Macaulay and \eqref{eq:resolOJ} ends at
level $\kappa$, then $R=R_\kappa$ and
$\debar R=0$. By \cite[Theorem~1.1]{AW1} we have that if $h\in \hol_D$ then 
\begin{equation}\label{dualprincip}
h R = 0 \quad \textrm{if and only if} \quad h \in \J_V.
\end{equation}

\begin{example}\label{hy}
Let $V=f^{-1}(0)$ be a hypersurface in $D$. Then $0\to \hol(E_1)\stackrel{f}{\longrightarrow}\hol(E_0)$
is a resolution of $\hol/\langle f \rangle$, where $E_1$ and $E_0$ are auxiliary trivial line bundles.
The associated current $U$ then becomes $(1/f)\otimes e_1$, where $e_1$ is a holomorphic frame for $E_1$,
and the associated residue current $R$ is $\debar(1/f)\otimes e_1$.

Let $g_1,\ldots,g_\kappa\in \hol(D)$ be a regular sequence. Then the Koszul complex associated to the 
$g_j$ is a free resolution of $\hol_D/\langle g_1,\ldots,g_\kappa\rangle$. The associated residue
current $R$ then becomes the Coleff-Herrera product 
\begin{equation*}
\debar \frac{1}{g_1}\wedge \cdots \wedge \debar \frac{1}{g_\kappa}, 
\end{equation*}
introduced in \cite{CH}, 
times an auxiliary frame element, see \cite[Theorem~1.7]{ABullSci}. 
\hfill $\qed$
\end{example}

\subsection{Structure forms of a complex space}\label{ssec:structure}
Assume first that $V$ is a reduced hypersurface, i.e., $V=f^{-1}(0)\subset D\subset \C^N$, $N=n+1$, 
where $f\in \hol(D)$ and $df\neq 0$ on $V_{reg}$.
Let $\omega'$ be a meromorphic $(n,0)$-form in $D\subset \C^{n+1}_z$ such that 
\begin{equation*}
df\wedge \omega' = 2\pi i \, dz_1\wedge \cdots \wedge dz_{n+1} \quad \textrm{on} \quad V_{reg}.
\end{equation*}
Then $\omega:=i^*\omega'$, where $i\colon V\hookrightarrow D$ is the inclusion, is a meromorphic form on $V$ that is uniquely
determined by $f$; $\omega$ is the Poincar\'e residue of the meromorphic form 
$2\pi i dz_1\wedge \cdots \wedge dz_{n+1}/f(z)$. For brevity we will sometimes write $dz$ for 
$dz_1\wedge \cdots \wedge dz_N$. Leray's residue formula can be formulated as
\begin{equation}\label{leila}
\int \debar \frac{1}{f} \wedge dz \wedge \xi = 
\lim_{\epsilon\to 0}\int_V \chi(|h|/\epsilon)\omega\wedge i^* \xi,
\end{equation}
where $\xi$ is a $(0,n)$-test form in $D$, the left hand side is the action of $\debar(1/f)$ on 
$dz\wedge \xi$ and $h$ is a holomorphic tuple such that $\{h=0\}=V_{sing}$. 
If we consider $\omega$ as a meromorphic current on $V$ we can rephrase \eqref{leila} as
\begin{equation}\label{Leray}
\debar\frac{1}{f}\wedge dz = i_* \omega.
\end{equation}

\smallskip

Assume now that $V\stackrel{i}{\hookrightarrow} D\subset \C^N$ is an arbitrary pure $n$-dimensional analytic 
subset. From Section~\ref{ssec:rescurr} we have, 
given a free resolution \eqref{eq:resolOJ} of $\hol_D/\J_V$ and a choice of Hermitian metrics on the involved 
bundles $E_j$, the associated residue current $R$ that plays the role
of $\debar(1/f)$. By the following result, which is an abbreviated version of \cite[Proposition~3.3]{AS}, 
there is an almost semi-meromorphic current $\omega$ on $V$ such that $R\wedge dz=i_*\omega$; such a current
will be called a \emph{structure form} of $V$.

\begin{proposition}\label{structureprop}
Let \eqref{eq:resolOJ} be a Hermitian free resolution of $\hol_D/\J_V$ in $D$ and let $R$ be the associated residue
current. Then there is a unique almost semi-meromorphic current
\begin{equation*}
\omega=\omega_0+\omega_1 +\cdots + \omega_{n-1}
\end{equation*}
on $V$, where $\omega_r$ is smooth on $V_{reg}$, has bidegree $(n,r)$, and takes values in $E_{\kappa+r}|_V$, such that 
\begin{equation}\label{strut}
R\wedge dz_1\wedge \cdots \wedge dz_N = i_* \omega.
\end{equation}
Moreover, 
\begin{equation*}
f_\kappa|_V \omega_0 =0, \quad f_{\kappa+r}|_V \omega_r = \debar\omega_{r-1}, \quad r\geq 1,
\end{equation*} 
in the sense of currents on $V$, and 
there are $(0,1)$-forms $\alpha_k$, $k\geq 1$, that are smooth outside $V^k$ and that take values in 
$\textrm{Hom}(E_{\kappa+k-1}|_V,E_{\kappa+k}|_V)$, such that 
\begin{equation*}
\omega_k=\alpha_k \omega_{k-1}, \quad k\geq 1.
\end{equation*}
\end{proposition}

It is sometimes useful to reformulate \eqref{strut} suggestively as
\begin{equation}\label{strut2}
R\wedge dz_1\wedge \cdots \wedge dz_N = \omega\wedge [V],
\end{equation}
where $[V]$ is the current of integration along $V$.

The following result will be useful for us when defining our dualizing complex.

\begin{proposition}[Lemma~3.5 in \cite{AS}]\label{omegadiv}
If $\psi$ is a smooth $(n,q)$-form on $V$, then there is a smooth $(0,q)$-form $\psi'$ on $V$ with 
values in $E_p^*|_V$ such that $\psi=\omega_0\wedge \psi'$.
\end{proposition}

\subsection{Koppelman formulas in $\C^N$}\label{ssec:koppelCn}
We recall some basic constructions from \cite{AintrepI} and \cite{AintrepII}.
Let $D\Subset\C^N$ be a domain (not necessarily pseudoconvex at this point), let $k(z,\zeta)$ be an integrable
$(N,N-1)$-form in $D\times D$, and let $p(z,\zeta)$ be a smooth $(N,N)$-form in $D\times D$. Assume that
$k$ and $p$ satisfy the equation of currents 
\begin{equation}\label{currentKoppel}
\debar k(z,\zeta) = [\Delta^D] - p(z,\zeta)
\end{equation}
in $D\times D$, where $[\Delta^D]$ is the current of integration along the diagonal.
Applying this current equation to test forms $\psi(z)\wedge \varphi(\zeta)$ it is straightforward to verify that for
any compactly supported $(p,q)$-form $\varphi$ in $D$ one has the following Koppelman formula
\begin{equation*}
\varphi(z) = \debar_{z} \int_{D_{\zeta}} k(z,\zeta)\wedge \varphi(\zeta) + 
\int_{D_{\zeta}} k(z,\zeta)\wedge \debar\varphi(\zeta) + \int_{D_{\zeta}} p(z,\zeta)\wedge \varphi(\zeta).
\end{equation*}

In \cite{AintrepI} Andersson introduced a very flexible method of producing solutions to \eqref{currentKoppel}.
Let $\eta=(\eta_1,\ldots,\eta_N)$ be a holomorphic tuple in $D\times D$ that defines the diagonal and let 
$\Lambda_{\eta}$ be the exterior algebra spanned by $T^*_{0,1}(D \times D)$
and the $(1,0)$-forms $d\eta_1,\ldots,d\eta_N$. On forms with values in $\Lambda_{\eta}$ interior multiplication
with $2\pi i \sum\eta_j \partial/\partial \eta_j$, denoted $\delta_{\eta}$, is defined; put $\nabla_{\eta}=\delta_{\eta}-\debar$.

Let $s$ be a smooth $(1,0)$-form in $\Lambda_{\eta}$ such that $|s|\lesssim |\eta|$
and $|\eta|^2\lesssim |\delta_{\eta}s|$ and let $B=\sum_{k=1}^N s\wedge (\debar s)^{k-1}/(\delta_{\eta}s)^k$.
It is proved in \cite{AintrepI} that then $\nabla_{\eta} B = 1-[\Delta^D]$. Identifying terms of top degree
we see that $\debar B_{N,N-1} = [\Delta^D]$ and we have found a solution to \eqref{currentKoppel}.
For instance, if we take $s=\partial |\zeta-z|^2$ and $\eta=\zeta-z$, then the resulting $B$ is sometimes called
the full Bochner-Martinelli form and the term of top degree is the classical Bochner-Martinelli kernel. 

A smooth section $g(z,\zeta)=g_{0,0}+\cdots +g_{N,N}$ of $\Lambda_{\eta}$, defined for $z\in D_1\subset D$
and $\zeta\in D_2\subset D$, such that $\nabla_{\eta} g=0$ and
$g_{0,0}|_{\Delta^D\cap D'} = 1$, where $D':=D_1\cap D_2$, is called a
\emph{weight} in $D_1\times D_2$.
 It follows that $\nabla_{\eta} (g\wedge B) = g-[\Delta^D]$ and, identifying 
terms of bidegree $(N,N-1)$, we get that
\begin{equation}\label{gulp}
\debar (g\wedge B)_{N,N-1} = [\Delta^D] - g_{N,N}
\end{equation} 
in $D'\times D'$.  Hence $(g\wedge B)_{N,N-1}$ and $g_{N,N}$ give a
solution to \eqref{currentKoppel} in $D'\times D'$.

If $D$ is pseudoconvex and $K$ is a holomorphically convex compact
subset, then one can find a weight $g$ in $D'\times D$ for some 
neighborhood $D'\subset D$ of $K$ such that $z\mapsto g(z,\zeta)$ is holomorphic 
in $D'$, which in particular means that there are no differentials of
the form 
$d\bar z_j$, and $\zeta \mapsto g(z,\zeta)$ has compact support in $D$;
see, e.g., Example~2 in \cite{AintrepII}.

\subsection{Koppelman formulas for $(0,q)$-forms on a complex space}\label{ssec:koppel0qX}
We briefly recall from \cite{AS} the construction of Koppelman formulas for $(0,q)$-forms 
on $V\subset D$. The basic idea is to use the currents $U$ and $R$ discussed in Section~\ref{ssec:rescurr} 
to construct a weight that will yield an integral formula of division/interpolation type in the same spirit
as in, e.g., \cite{BoB, P}. 

\smallskip

Let \eqref{eq:resolOJ} be a resolution of $\hol_{D}/\J_V$, where as before $\J_V$ is the sheaf in
$D$ associated to $V\stackrel{i}{\hookrightarrow} D$. 
One can find, see \cite[Proposition~5.3]{AintrepII}, holomorphic 
$\Lambda_\eta$-valued Hefer morphisms $H_k^{\ell}\colon E_k\to
E_{\ell}$ of bidegree $(k-\ell,0)$ such that 
$H_k^k=I_{E_k}$ and 
\begin{equation}\label{retur}
\delta_{\eta} H_k^{\ell} = H_{k-1}^{\ell} f_k(\zeta) - f_{\ell+1}(z)H_k^{\ell+1}, \quad k>1.
\end{equation}

Let $F$ be  a holomorpic
tuple in $D$ such that $\{F=0\}=V$, let 
$U^{\epsilon}=\chi(|F|/\epsilon) u$, and let 
\begin{equation*}
R^\epsilon:=1-\sum f_k U_k^\epsilon + \dbar U^\epsilon, 
\end{equation*}
so that $R^\epsilon=\sum_{k\geq 0} R_k^\epsilon$, where $R_0^\epsilon=1-
\chi(|F|/\epsilon)$ and 
$R_k^{\epsilon}=\dbar\chi(|F|/\epsilon) \wedge u$ for
$k\geq 1$. 
Then $\lim_{\epsilon\to 0}U^\epsilon = U$ and $\lim_{\epsilon\to
  0}R^\epsilon = R$, cf.\ \eqref{gulU} and \eqref{lajbans}, and 
moreover
\begin{equation}\label{glambda}
\gamma^{\epsilon} := \sum_{k=0}^N H^0_k R^{\epsilon}_k(\zeta) + f_1(z) \sum_{k=1}^N H^1_k U^{\epsilon}_k(\zeta).
\end{equation}
is a weight in $D'\times D'$ for $\epsilon > 0$.
Let $g$ be an arbitrary weight in $D'\times D'$. Then
$\gamma^{\epsilon}\wedge g$ is again a weight in $D'\times D'$ and we get
\begin{equation}\label{plast}
\debar (\gamma^{\epsilon}\wedge g\wedge B)_{N,N-1} = [\Delta^D] - (\gamma^{\epsilon}\wedge g)_{N,N}
\end{equation}
in the current sense in $D'\times D'$, cf.\ \eqref{gulp}. Let us 
proceed formally and, also, let us temporarily assume that $V$ is
Cohen-Macaulay and that \eqref{eq:resolOJ} ends at level $\kappa$, so that $R$ is $\debar$-closed.
Then, multiplying \eqref{plast} with $R(z)\wedge dz$
and using \eqref{dualprincip} so that $f_1(z)R(z)=0$, we get that
\begin{equation}\label{tre}
\debar \left(R(z)\wedge dz\wedge (HR^{\epsilon}(\zeta)\wedge g\wedge B)_{N,N-1}\right) = 
R(z)\wedge dz\wedge [\Delta^D] - R(z)\wedge dz\wedge(HR^{\epsilon}(\zeta)\wedge g)_{N,N},
\end{equation}
where $HR^{\epsilon} = \sum_{k=0}^N H^0_k R^{\epsilon}_k$, cf.\ \eqref{glambda}.
In view of \eqref{strut2} we have $R(z)\wedge dz\wedge [\Delta^D]=\omega\wedge [\Delta^V]$,
where $[\Delta^V]$ is the integration current along the diagonal $\Delta^V \subset V\times V \subset D\times D$,
and formally letting $\epsilon\to 0$ in \eqref{tre} we thus get
\begin{equation}\label{plutt}
\debar \Big( \omega(z)\wedge [V_z]\wedge (HR(\zeta)\wedge g\wedge B)_{N,N-1} \Big) = 
\omega\wedge [\Delta^V] - \omega(z)\wedge [V_z]\wedge (HR(\zeta)\wedge g)_{N,N}.
\end{equation}
To see what this means we will use \eqref{strut2}. Notice first that,
since $H$, $R$, $g$, and $B$ takes values in $\Lambda_\eta$, one can factor out $d\eta=d\eta_1\wedge \cdots \wedge d\eta_N$
from $(HR(\zeta)\wedge g\wedge B)_{N,N-1}$ and $(HR(\zeta)\wedge g)_{N,N}$.
After making these factorization in \eqref{plutt} we may replace $d\eta$ by 
$C_{\eta}(z,\zeta)d\zeta$, where $C_{\eta}(z,\zeta)=N!\det(\partial \eta_j/\zeta_k)$, 
since $\omega(z)\wedge [V_z]$ has full degree in $dz_j$. More precisely, 
let $\epsilon_1,\ldots,\epsilon_N$ be a basis for an auxiliary trivial complex vector bundle over $D\times D$
and replace all occurrences of 
$d\eta_j$ in $H$, $g$, and $B$ by $\epsilon_j$. Denote the resulting forms
by $\hat{H}$, $\hat{g}$, and $\hat{B}$ respectively and let
\begin{equation}\label{eq:k}
k(z,\zeta)= C_{\eta}(z,\zeta)\epsilon_N^*\wedge \cdots \wedge \epsilon_1^* \lrcorner
\sum_{k=0}^n \hat{H}_{p+k}^0 \omega_k(\zeta) \wedge (\hat{g}\wedge \hat{B})_{n-k, n-k-1}
\end{equation}
\begin{equation}\label{eq:p}
p(z, \zeta) = C_{\eta}(z,\zeta)\epsilon_N^*\wedge \cdots \wedge \epsilon_1^* \lrcorner
\sum_{k=0}^n \hat{H}_{p+k}^0 \omega_k(\zeta) \wedge \hat{g}_{n-k, n-k}.
\end{equation}
Notice that $k$ and $p$ have bidegrees $(n,n-1)$ and $(n,n)$ respectively.
In view of \eqref{strut2} we can replace $(HR\wedge g\wedge B)_{N,N-1}$ and $(HR\wedge g)_{N,N}$ 
with $[V_{\zeta}]\wedge k(z,\zeta)$ and $[V_{\zeta}]\wedge p(z,\zeta)$
respectively in \eqref{plutt}. It follows that 
\begin{equation*}
\debar \big (\omega(z)\wedge k(z,\zeta)\big ) = \omega\wedge [\Delta^V] - \omega(z)\wedge p(z,\zeta)
\end{equation*}
holds in the current sense at least on $V_{reg}\times V_{reg}$.
The formal computations above can be made rigorous, see \cite[Section~5]{AS}, and
combined with Proposition~\ref{omegadiv} we get Proposition~\ref{koppelcurrent} below;
notice that $\omega=\omega_0$ and $\debar\omega=0$ since we are assuming that 
$V$ is Cohen-Macaulay and that \eqref{eq:resolOJ} ends at level $\kappa$.
 
The following result will be the starting point of the next section and it holds without
any assumption about Cohen-Macaulay.

\begin{proposition}[Lemma~5.3 in \cite{AS}]\label{koppelcurrent}
With $k(z,\zeta)$ and $p(z,\zeta)$ defined by \eqref{eq:k} and \eqref{eq:p} respectively we have
\begin{equation*}
\debar k(z, \zeta) = [\Delta^V] - p(z,\zeta)
\end{equation*}
in the sense of currents on $V_{reg}\times V_{reg}$.
\end{proposition}

\begin{remark}
In \cite{AS} it is assumed that 
$g$ is a weight in $D'\times D$, where $D'\Subset D$ and $\zeta\mapsto
g(z,\zeta)$ has compact support in $D$, 
but the proof goes through for any weight.
\end{remark}

\smallskip

The integral operators $\mathscr{K}$ and $\mathscr{P}$ for forms in $\W^{0,q}$ introduced in \cite{AS} are defined as follows.
Let $g$ 
in \eqref{eq:k} and \eqref{eq:p} be a weight in $D'\times D$, where
$D'\Subset D$ and $\zeta\mapsto g(z,\zeta)$ has compact support in
$D$, cf.\ Section~\ref{ssec:koppelCn}, 
and let $\mu\in \W^{0,q}(D)$. Since $\omega$ and $B$ are almost semi-meromorphic 
$k(z,\zeta)$ and $p(z,\zeta)$ are also almost semi-meromorphic and it follows from 
Proposition~\ref{multprop} that $k(z,\zeta)\wedge \mu(\zeta)$ and $p(z,\zeta)\wedge\mu(\zeta)$ are in $\W(V'\times V)$, 
where $V'=D'\cap V$. 
Let $\tilde{\pi} \colon V'_z\times V_{\zeta} \to V'_z$ be the natural
projection onto $V'_z$. 
It follows that
\begin{equation*}
\mathscr{K} \mu (z) := \tilde{\pi}_* \big(k(z,\zeta)\wedge \mu(\zeta)\big),
\end{equation*}
\begin{equation*}
\mathscr{P} \mu (z) := \tilde{\pi}_* \big(p(z,\zeta)\wedge \mu(\zeta)\big),
\end{equation*}
are in $\W(V'_z)$.
The sheaves $\A_V^{0,\bullet}$ are then morally defined to be the smallest sheaves that contain
$\mathcal{E}_V^{0,\bullet}$ and are closed under operators $\mathscr{K}$ and under multiplication
with $\mathcal{E}_V^{0,\bullet}$. More precisely, the stalk $\A_{V,x}^{0,q}$ consists of those germs of currents
which can be written as a finite sum of of terms 
\begin{equation*}
\xi_m\wedge \K_m\big(\cdots \xi_1\wedge \K_1(\xi_0)\cdots\big),
\end{equation*}
where $\xi_j$ are smooth $(0,*)$-forms and $\K_j$ are integral operators at $x$ of the above form; 
cf.\ \cite[Definition~7.1]{AS}.

\section{Koppelman formulas for $(n,q)$-forms}\label{koppelsec}

Let $V$ be a pure $n$-dimensional analytic subset of a pseudoconvex domain $D\subset \C^N$ and 
let $\omega$ be a structure form on $V$. Let 
$g$ be a weight in $D\times D'$, where $D'\subset D$ and let 
$k(z,\zeta)$ and $p(z,\zeta)$ be the kernels defined 
respectively in \eqref{eq:k} and \eqref{eq:p}.
Since $k$ and $p$ are almost semi-meromorphic it follows from Proposition~\ref{multprop} that if
$\mu=\mu(z)\in \mathcal{W}^{n,q}(V)$, then 
$k(z,\zeta)\wedge \mu(z)$ and $p(z,\zeta)\wedge \mu(z)$ are well-defined currents
in $\mathcal{W}(V\times V)$.
Assume that $z\mapsto g(z,\zeta)$ has compact support in $D$ or that
$\mu$ has 
compact support in $V$.
Let $\pi \colon V_{z}\times V'_{\zeta} \to V'_{\zeta}$ be the natural
projection, where, as above, $V'=D'\cap V$, and define
\begin{equation}\label{tomteK}
\check{\mathscr{K}}\mu (\zeta) := \pi_* \big(k(z,\zeta)\wedge \mu(z)\big)
\end{equation}
\begin{equation}\label{tomteP}
\check{\mathscr{P}}\mu (\zeta) := \pi_* \big(p(z,\zeta)\wedge \mu(z)\big).
\end{equation}
It follows that $\check{\mathscr{K}} \mu$ and $\check{\mathscr{P}} \mu$ are well-defined currents in 
$\mathcal{W}(V'_\zeta)$. Notice that $\check{\mathscr{P}} \mu$ is of the form $\sum_r\omega_r\wedge \xi_r$, where
$\xi_r$ is a smooth $(0,*)$-form (with values in an appropriate bundle) in general, and holomorphic
if the weight $g(z,\zeta)$ is chosen holomorphic in $\zeta$; cf.\ \eqref{eq:p}. It is natural to write
\begin{equation*}
\check{\mathscr{K}}\mu (\zeta) = \int_{V_z} k(z,\zeta)\wedge \mu(z), \quad
\check{\mathscr{P}}\mu (\zeta) = \int_{V_z} p(z,\zeta)\wedge \mu(z).
\end{equation*}

We have the following analogue of Proposition~6.3 in \cite{AS}.
\begin{proposition}\label{Wpropp}
Let $\mu(z)\in \mathcal{W}^{n,q}(V)$ and assume that $\debar \mu \in
\mathcal{W}^{n,q+1}(V)$.
Let $\check{\mathscr{K}}$ and $\check{\mathscr{P}}$ be as above. 
Then 
\begin{equation}\label{hurv}
\mu = \debar \check{\mathscr{K}} \mu + \check{\mathscr{K}} (\debar \mu) + \check{\mathscr{P}} \mu
\end{equation}
in the sense of currents on $V'_{reg}$.
\end{proposition}

\begin{proof}
If $\varphi=\varphi(\zeta)$ is a $(0,n-q)$-test form on $V'_{reg}$ it follows, cf.\ the beginning of Section~\ref{ssec:koppelCn}, 
from Proposition~\ref{koppelcurrent} that
\begin{equation}\label{gaston}
\varphi(z)= \debar_z \int_{V'_{\zeta}} k(z,\zeta)\wedge \varphi(\zeta) +
\int_{V'_{\zeta}} k(z,\zeta)\wedge \debar \varphi(\zeta) +
\int_{V'_{\zeta}} p(z,\zeta)\wedge \varphi(\zeta)
\end{equation}
for $z\in V'_{reg}$.
By \cite[Lemma~6.1]{AS}\footnote{The proof goes through also in our setting, i.e., when
$g$ not necessarily has compact support in $D_{\zeta}$ but
$\varphi(\zeta)$ has.} 
the first two terms on the right hand side are smooth on $V'$. The last term
is smooth $V'$ since $z\mapsto p(z,\zeta)$ is smooth. 
Assume that $z\mapsto g(z,\zeta)$ has compact support in $D$. Then so have
$z\mapsto k(z,\zeta)$ and $z\mapsto p(z,\zeta)$. Thus each term in the
right hand side of \eqref{gaston} is a test form in $z$, and so $\mu$
acts on each term. Thus \eqref{hurv} follows in this case. If $\mu$
has compact support \eqref{hurv} holds without the assumption that
$z\mapsto g(z,\zeta)$ has compact support.

For the general case, let $h=h(z)$ be a holomorphic tuple
such that $\{h=0\}=V_{sing}$ and let $\chi_{\epsilon}=\chi(|h|/\epsilon)$.
Then the proposition holds for $\chi_{\epsilon} \mu$ (since $k$ and $p$ have compact support in $z$).
Since $k(z,\zeta)\wedge \mu(z)$ and $p(z,\zeta)\wedge \mu(z)$ are in $\mathcal{W}(V'\times V)$
it follows that $\check{\mathscr{K}} (\chi_{\epsilon}\mu)\to \check{\mathscr{K}} \mu$ and that
$\check{\mathscr{P}} (\chi_{\epsilon}\mu) \to \check{\mathscr{P}} \mu$ in the sense of currents, and consequently
$\debar\check{\mathscr{K}} (\chi_{\epsilon}\mu) \to \debar\check{\mathscr{K}} \mu$ in the current sense.
It remains to see that 
$\lim_{\epsilon\to 0}\check{\mathscr{K}} (\debar(\chi_{\epsilon} \mu))=\check{\mathscr{K}}(\debar\mu)$.
In fact, since by assumption $\debar \mu \in \mathcal{W}(V)$ it follows 
that $\check{\mathscr{K}} (\chi_{\epsilon}\debar\mu)\to \check{\mathscr{K}} (\debar \mu)$ and so
\begin{equation}\label{eq:urk0}
\lim_{\epsilon\to 0} \check{\mathscr{K}} (\debar(\chi_{\epsilon} \mu))=
\check{\mathscr{K}}(\debar\mu) + \lim_{\epsilon \to 0} \check{\mathscr{K}}(\debar\chi_{\epsilon}\wedge \mu);
\end{equation}
it also follows that
\begin{equation}\label{eq:urk} 
\debar \chi_{\epsilon}\wedge \mu = \debar (\chi_{\epsilon} \mu) - \chi_{\epsilon} \debar \mu\to \debar \mu - \debar \mu=0.
\end{equation}
Now, if $\zeta$ is in a compact subset of $V'_{reg}$ and $\epsilon$ is sufficiently small, 
then $k(z,\zeta)\wedge \debar \chi_{\epsilon}(z)$ is a smooth form times $\omega=\omega(\zeta)$.
Since $\mu(z)\wedge \omega(\zeta)$ is just a tensor product it follows from \eqref{eq:urk} that
$\debar \chi_{\epsilon}(z)\wedge \mu(z) \wedge \omega(\zeta)\to 0$.
Hence, $\check{\mathscr{K}}(\debar\chi_{\epsilon}\wedge \mu)\to 0$ as a current on 
$V'_{reg}$ and so by \eqref{eq:urk0} we have 
$\lim_{\epsilon\to 0}\check{\mathscr{K}} (\debar(\chi_{\epsilon} \mu))=\check{\mathscr{K}}(\debar\mu)$. 
\end{proof}

\section{The dualizing Dolbeault complex of $\B^{n,q}_X$-currents}\label{Asection}
Let $X$ be a reduced complex space of pure dimension $n$.
We define our sheaves $\B_X^{n,\bullet}$ in a way similar to the definition of $\A_X^{0,\bullet}$;
see the end of Section~\ref{ssec:koppel0qX}.
In a moral sense $\oplus_q\B_X^{n,q}$ then becomes the smallest sheaf that contains 
$\oplus_q\mathcal{E}_X^{n,q}$ and that is closed under integral operators $\check{\mathscr{K}}$ and exterior
products with elements of $\oplus_q\mathcal{E}_X^{0,q}$.  

\begin{definition}\label{A-sheaves}
We say that an $(n,q)$-current $\psi$ on an open set $V\subset X$ is a section of $\B_X^{n,q}$, 
$\psi\in \B^{n,q}(V)$, if, for every $x\in V$, the germ $\psi_x$ can be written as a finite sum of terms
\begin{equation}\label{Acurrent}
\xi_m\wedge \check{\mathscr{K}}_m\left( \cdots \xi_1\wedge \check{\mathscr{K}}_1(\omega\wedge \xi_0) \cdots \right),
\end{equation}
where $\xi_j$ are smooth $(0,*)$-forms, $\check{\mathscr{K}}_j$ are integral operators at $x$ given by \eqref{tomteK}
with kernels of the form \eqref{eq:k}, and $\omega$ is a structure form at $x$.
\end{definition}

Notice that $\omega$ takes values in some bundle $\oplus_j E_j$ so we let $\xi_0$ take values in $\oplus_j E^*_j$ 
to make $\omega\wedge \xi_0$ scalar valued.

\smallskip

It is clear that $\check{\mathscr{K}}$ preserves $\oplus_q\B_X^{n,q}$. 
Notice that we allow $m=0$ in the definition above so that $\B_X^{n,\bullet}$ contains all 
currents of the form $\omega\wedge \xi_0$, where $\xi_0$ is smooth with values in $\oplus_jE^*_j$.
Since $\check{\mathscr{P}}\mu$ is of the form $\omega\wedge \xi$ for a smooth $\xi$,
also $\check{\mathscr{P}}$ preserves $\oplus_q\B_X^{n,q}$.

Recall that if $\mu\in \W^{n,*}(V)$, then $\check{\mathscr{K}} \mu \in \W^{n,*}(V')$, where $V'$ is 
a relatively compact subset of $V$. Since $\omega\wedge \xi_0\in \W_X^{n,*}$ it follows that 
$\B_X^{n,q}$ is a subsheaf of $\W_X^{n,q}$. In fact, by Proposition~\ref{prop:A} below we can say more. 

\begin{definition}\label{def:domain}
A current $\mu\in \oplus_q\W_X^{n,q}$ is said to be in {\em the domain of $\debar$}, $\mu \in \textrm{Dom}\,\debar$, 
if $\debar \mu \in \oplus_q\W_X^{n,q}$.
\end{definition}

Assume that $\mu\in \W_X^{n,q}$ is smooth on $X_{reg}$, let $h$ be a
holomorphic tuple such that $\{h=0\}=X_{sing}$, and, as above, let
$\chi_{\epsilon}=\chi(|h|/\epsilon)$. Then $\debar (\chi_{\epsilon}\mu)\to \debar \mu$ since $\mu$ has the SEP. 
In view of the first equality in \eqref{eq:urk} it follows that $\debar \mu$ has the SEP if and only if
$\debar\chi_{\epsilon}\wedge \mu\to 0$ as $\epsilon\to 0$; this last condition 
can be interpreted as a ``boundary condition'' on $\mu$ at $X_{sing}$.

\begin{proposition}\label{prop:A}
Let $X$ be a reduced complex space of pure dimension $n$. Then
\begin{itemize}
\item[(i)] $\B_X^{n,q}\big|_{X_{reg}} = \mathcal{E}_X^{n,q}\big|_{X_{reg}}$,
\item[(ii)] $\mathcal{E}_X^{n,q} \subset \B_X^{n,q} \subset \textrm{Dom}\, \debar$.
\end{itemize}
\end{proposition}

To prove $(i)$ we need to prove that if $\mu\in\W(V)$ is smooth in a
neighborhood of a given point $x\in V'_{reg}$, then $\check{\mathscr{K}}
\mu (z)$ is smooth in a neighborhood of $x$. This is proved in the
same way as part (i) of Lemma~6.1 in \cite{AS}. 
The proof (of the second inclusion) of $(ii)$ is similar to the proof
that $\A_X^{0,q}\subset \textrm{Dom}\, \debar$ in \cite{AS}, see Section~7 and 
Lemmas~6.4 and~4.1 in \cite{AS}. We include a proof for the reader's
convenience. 

\begin{proof}[Proof of $(ii)$]
Let $\psi$ be a smooth $(n,q)$-form on $X$ and let $\omega=\sum_r\omega_r$ be a structure form. 
Then, by Proposition~\ref{omegadiv}, there is smooth $(0,q)$-form $\xi$ (with values in the appropriate bundle) 
such that $\psi=\omega_0\wedge \xi$ and so $\mathcal{E}_X^{n,q} \subset \B_X^{n,q}$. 

To prove the second inclusion of (ii)
we may assume that $\mu$ is of the form \eqref{Acurrent}. Let $k_j(w^{j-1},w^j)$, $j=1,\ldots,m$, 
be the integral kernel corresponding to $\check{\K}_j$; $w^j$ are coordinates on $V$ for each $j$.
We define an almost semi-meromorphic current $T$ on $V^{m+1}$ (the $m+1$-fold Cartesian product) by
\begin{equation}\label{hund-1}
T:=\bigwedge_{j=1}^m k_j(w^{j-1},w^j)\wedge \omega(w^0), 
\end{equation}
and we let $T_r$ be the term of $T$ corresponding to $\omega_r$.
Notice that $\pi_*(\xi\wedge T)=\mu$ for a suitable smooth $(0,*)$-form $\xi$
on $V^{m+1}$, where $\pi\colon V^{m+1}\to V_{w^m}$ is the natural projection.
We claim that 
\begin{equation}\label{klot}
\lim_{\epsilon\to 0}\debar \chi(|h(w^m)|/\epsilon)\wedge T_r=0
\end{equation}
for all $r$, where $h$ is a holomorphic tuple such that
$\{h=0\}=V_{sing}$. 
Taking this for granted, 
\[
\lim_{\epsilon\to 0} \dbar\chi_\epsilon\wedge \mu=
\pi_*\big (\lim_{\epsilon\to 0} \dbar \chi(|h(w^m)|/\epsilon) \wedge \xi\wedge
T\big )=0, 
\]
and thus $\mu\in \textrm{Dom}\, \debar$, cf.\ the discussion after
Definition~\ref{def:domain}.

We will prove that \eqref{klot} holds for all $r$ by double induction
over  $m$ and $r$. 
If $m=0$ then $T=\omega(w^0)$ and, since $\debar\omega_r=f_{r+1}|_V \omega_{r+1}$ by \eqref{structureprop},
it follows that $\debar T$ has the SEP, i.e., $\lim_{\epsilon\to 0}\debar \chi(|h|/\epsilon)\wedge T=0$.

Assume that \eqref{klot} holds for $m\leq k-1$ and all $r$. The left hand side of \eqref{klot}, with $m=k$,
defines a pseudomeromorphic current $\tau_r$ of bidegree $(*,kn-k+r+1)$ since each $k_j$ has bidegree $(*,n-1)$
and clearly $\textrm{supp}\, \tau_r\subset Sing(V_{w^m})\times V^m$.
If $w^j\neq w^{j-1}$, then $k_j(w^{j-1},w^j)$ is a smooth form times some structure form
$\tilde{\omega}(w^j)$. Thus $T$, with $m=k$, is a smooth form times the \emph{tensor product} of two currents, 
each of which is of the form \eqref{hund-1} with $m<k$. By the induction hypothesis, it follows that
\eqref{klot}, with $m=k$, holds outside $\{w^j=w^{j-1}\}$ for all $j$. Hence,
$\tau_r$ has support in $\{w^1=\cdots =w^k\}\cap (Sing(V_{w^m})\times V^m)$, which has codimension at least
$kn+1$ in $V^{k+1}$. Since $\tau_0$ has bidegree $(*,kn-k+1)$, $k\geq 1$, it follows from the dimension principle that
$\tau_0=0$. 

By Proposition~\ref{structureprop}, there is a $(0,1)$-form $\alpha_1$ such that $\omega_1=\alpha_1\omega_0$ and 
$\alpha_1$ is smooth outside $V^1$ (cf.\ \eqref{ballong}) which has codimension at least $2$ in $V$.
Since $\tau_1=\alpha_1(w^0)\tau_0$ outside $V^1_{w^0}$ and $\tau_0=0$ it follows that $\tau_1$ has support in 
$\{w^1=\cdots =w^k\}\cap (V^1_{w^0}\times V^m)$. This set has codimension at least $kn+2$ in $V^{m+1}$
and $\tau_1$ has bidegree $(*,kn-k+2)$ so the dimension principle shows that $\tau_1=0$. Continuing in this way
we get that $\tau_r=0$ for all $r$ and hence, \eqref{klot} holds with $m=k$. 
\end{proof}

\begin{theorem}\label{thm:debar}
Let $X$ be a reduced complex space of pure dimension $n$. Then
$\debar \colon \B_X^{n,q}\to \B_X^{n,q+1}$.
\end{theorem}
\begin{proof}
Let $\psi$ be a germ of a current in $\B_X^{n,q}$ at some point $x$;
we may assume that 
\begin{equation*}
\psi=\xi_m\wedge \check{\mathscr{K}}_m\left( \cdots \xi_1\wedge \check{\mathscr{K}}_1(\omega\wedge \xi_0) \cdots \right),
\end{equation*}
see Definition~\ref{A-sheaves}.

We will prove the theorem by induction over $m$. Assume first that $m=0$ so that $\psi=\omega\wedge \xi_0$;
recall that $\xi_0$ takes values in $\oplus_j E_j^*$ so that $\psi$ is scalar valued. Then,
by Proposition~\ref{structureprop}, we have that 
\begin{equation*}
\debar \psi = \debar \omega \wedge \xi_0 \pm \omega\wedge \debar \xi_0
= f\omega \wedge \xi_0 \pm \omega\wedge\debar \xi_0 
=\omega\wedge f^*\xi_0 \pm \omega\wedge \debar \xi_0,
\end{equation*}
where $f=\oplus_{r=0}^n f_{p+r}|_V$ and $f^*$ is the transpose of $f$. Hence, $\debar \psi$
is in $\B_X^{n,q+1}$.
Assume now that $\debar \psi'\in \oplus_q\B_X^{n,q}$, where
\begin{equation*}
\psi'=\xi_{m-1}\wedge \check{\mathscr{K}}_{m-1}\left( \cdots \xi_1\wedge \check{\mathscr{K}}_1(\omega\wedge \xi_0) \cdots \right).
\end{equation*}
Then $\psi'\in \textrm{Dom}\,\debar \subset \W_X$ and by Proposition~\ref{prop:A}
$\psi'$ is smooth on $X_{reg}$. Thus, from Proposition~\ref{Wpropp} it follows that 
\begin{equation}\label{penna}
\psi'=\debar \check{\mathscr{K}}_{m} \psi' + \check{\mathscr{K}}_{m}(\debar \psi') + \check{\mathscr{P}}_{m}\psi'
\end{equation} 
in the current sense on $V_{reg}$, where $V$ is some neighborhood of $x$. 
By the induction hypothesis, $\debar \psi'\in \oplus_q\B_X^{n,q}$ and since 
$\check{\mathscr{K}}_{m}$ and $\check{\mathscr{P}}_{m}$ preserve $\oplus_q\B_X^{n,q}$ 
and furthermore $\oplus_q\B_X^{n,q}\subset \textrm{Dom}\, \debar$ it follows that every term of \eqref{penna}
has the SEP. Thus, \eqref{penna} holds in fact on $V$. Finally, notice that
$\psi=\xi_m\wedge\check{\mathscr{K}}_m\psi'$ and so, since $\psi'$, $\check{\mathscr{K}}_{m}(\debar \psi')$, and 
$\check{\mathscr{P}}_{m}\psi'$ all are in $\oplus_q\B_X^{n,q}$, it follows that $\debar \psi \in \B_X^{n,q+1}$.
\end{proof}

\medskip

\begin{proof}[Proof of Theorem~\ref{thm:Homotopy}]
Choose a weight $g$ in $D\times D'$, where $D'\Subset D$, such that 
$z\mapsto g(z,\zeta)$ has compact support in $D$, cf.\
Section~\ref{ssec:koppelCn}. 
Let $k(z,\zeta)$ and $p(z,\zeta)$ be the kernels defined by
\eqref{eq:k} and \eqref{eq:p}, respectively, and let
$\check{\mathscr{K}}$ and $\check{\mathscr{P}}$ be the associated
integral operators.

Let $\psi\in \B^{n,q}(V)$. By Proposition~\ref{Wpropp},
\begin{equation}\label{penna2}
\psi=\debar \check{\mathscr{K}} \psi + \check{\mathscr{K}}(\debar \psi) + \check{\mathscr{P}}\psi
\end{equation}
holds on $V'_{reg}$. Since $\check{\mathscr{K}}$ and  
$\check{\mathscr{P}}$ map $\oplus_q \B^{n,q}(V)$ to $\oplus_q \B^{n,q}(V')$
it follows from Theorem~\ref{thm:debar} that every term of \eqref{penna2} has the SEP. Hence, \eqref{penna2}
holds on $V'$ and the theorem follows.
\end{proof}

\medskip

\begin{proof}[Proof of Theorem~\ref{thm:resol}]
Let $V$ be a pure $n$-dimensional analytic subset of a pseudoconvex domain 
$D \subset \C^N$, let $\J_V$ be the sheaf in $D$ defined by $V$, 
let $i\colon V \hookrightarrow D$ be the inclusion, and, as above, let $\kappa=N-n$ be the codimension of $V$ in $D$. 
Let \eqref{eq:resolOJ} be a free resolution of $\hol_{D}/\J_V$ in (possibly a slightly smaller 
domain still denoted) $D$ and let $\omega=\sum_r\omega_r$ be an associated structure form.

Dualizing the complex \eqref{eq:resolOJ} 
and tensoring with the invertible sheaf
$\mathit{\Omega}_D^N$ gives the complex
\begin{equation}\label{eq:dualcplx}
0 \to \hol(E^*_0)\otimes_{\hol_D} \mathit{\Omega}_D^N \stackrel{f_1^*}{\longrightarrow} \cdots 
\stackrel{f_m^*}{\longrightarrow} \hol(E^*_m) \otimes_{\hol_D} \mathit{\Omega}_D^N \to 0.
\end{equation}
It is well-known that the cohomology sheaves of \eqref{eq:dualcplx} are isomorphic to 
$\Ext^{\bullet}(\hol_{D}/\J_V,\mathit{\Omega}_D^N)$ and that
$\Ext^{k}(\hol_{D}/\J_V,\mathit{\Omega}_D^N)=0$ for $k<\kappa$. Notice that if $V$ is Cohen-Macaulay,
i.e., if we can take $m=\kappa=\textrm{codim}\, V$ in \eqref{eq:resolOJ}, then 
$\Ext^{k}(\hol_{D}/\J_V,\mathit{\Omega}_D^N)=0$ for $k\neq \kappa$.

We define mappings $\varrho_k \colon \hol(E^*_{\kappa+k})\otimes \mathit{\Omega}_D^N \to \B_V^{n,k}$ by letting
$\varrho_k(h dz)=0$ for $k<0$ and $\varrho_k(h dz) = \omega_k \cdot h$ for $k\geq 0$;
here we let $\B_V^{n,k}:=0$ for $k< 0$ and $\hol(E_k^*)\otimes \Omega^N_D:=0$ for $k>m$.
We get a map
\begin{equation}\label{eq:diagr}
\varrho_{\bullet}\colon \left(\hol(E^*_{\kappa+\bullet})\otimes \mathit{\Omega}_D^N, f^*_{\kappa+\bullet}\right)
\longrightarrow \left(\B_V^{n,\bullet}, \debar \right)
\end{equation}
which is a morphism of complexes since if $h\in \hol(E_{\kappa+k}^*)$, then, by Proposition~\ref{structureprop},
\begin{equation*}
\debar \varrho_k (h dz) = \debar \omega_k \cdot h = f_{\kappa+k+1} \omega_{k+1}\cdot h = \omega_{k+1} \cdot f^*_{\kappa+k+1}h =
\varrho_{k+1}(f^*_{\kappa+k+1}h).
\end{equation*}
Hence, \eqref{eq:diagr} induces a map on cohomology. We claim that $\varrho_{\bullet}$ in fact is a quasi-isomorphism,
i.e., that $\varrho_{\bullet}$ induces an isomorphism on cohomology level. 
Given the claim it follows that $\HH^{k}(\B_V^{n,\bullet})$ is 
coherent since the corresponding cohomology sheaf of $(\hol(E^*_{\kappa+\bullet})\otimes \mathit{\Omega}_D^N, f^*_{\kappa+\bullet})$
is $\Ext^{\kappa+k}(\hol_{D}/\J_V,\mathit{\Omega}_D^N)$, which is coherent.

To prove the claim, recall first that $i_* \omega_k = R_k\wedge dz$. 
Thus, by \cite[Theorem~7.1]{ANoetherDual} the mapping on cohomology
is injective. For the surjectivity, choose 
a weight $g$ in $D\times D'$, where $D'\Subset D$, such that $g$ is
holomorphic in $\zeta$ and has compact support in $D_z$, cf.\
Section~\ref{ssec:koppelCn}, let
$k(z,\zeta)$ and $p(z,\zeta)$ be the integral kernels defined by
\eqref{eq:k} and \eqref{eq:p}, respectively, and let
$\check{\mathscr{K}}$ and $\check{\mathscr{P}}$ be the corresponding integral
operators. 
Let $\psi\in \B^{n,k}(V)$ be $\debar$-closed. By Theorem~\ref{thm:Homotopy} we get
\begin{equation*}
\psi(\zeta) = \debar \int_{V_z} k(z,\zeta)\wedge \psi(z) + 
\int_{V_z} p(z,\zeta)\wedge \psi(z)
\end{equation*}
in $V\cap D'$. Hence,
the $\debar$-cohomology class of $\psi$ is represented by the last
integral. 
Since $g$ is holomorphic in $\zeta$, the summand with index $k$ in
\eqref{eq:p} has exactly $n-k$ differentials of the form $d\bar z_j$ (and
$k$ differentials of the form $d\bar \zeta_j$). It
follows that  
\begin{multline*}
\int_{V_z} p(z,\zeta)\wedge \psi(z)=
 \int_{V_z} C_{\eta}(z,\zeta)\epsilon_N^*\wedge \cdots \wedge \epsilon_1^* \lrcorner
\hat{H}_{p+k}^0 \omega_k(\zeta) \wedge \hat{g}_{n-k, n-k}\wedge \psi(z)\\=:
\omega_k(\zeta) \wedge \int_{V_z} G(z,\zeta)\wedge \psi(z),
\end{multline*}
where $G$ takes values in $E_{p+k}^*$. Note that $G$ is holomorphic in
$\zeta$ since $g$ is. 
We will show that 
\begin{equation}\label{knaochta}
f_{p+k+1}^*\int_{V_z} G(z,\zeta)\wedge \psi(z)=0.
\end{equation}
Taking \eqref{knaochta} for granted, it follows that the class of $\psi$ is in the
image of the map on cohomology induced by $\varrho_k$, which proves the claim.

To prove \eqref{knaochta} first note that 
$d\eta \wedge G=H^0_{p+k}\wedge g_{n-k,n-k}$. 
By \eqref{retur}, 
\begin{multline}\label{umgas}
f_{p+k+1}^*H^0_{p+k} \wedge g_{n-k,n-k}=H^0_{p+k+1}f_{p+k+1}\wedge
g_{n-k,n-k}=\\
\delta_\eta H^0_{p+k+1} \wedge
g_{n-k,n-k} + 
f_1(z) H^1_{p+k}\wedge
g_{n-k,n-k}. 
\end{multline}
Since $H^0_{p+k+1}\wedge g_{n-k,n-k}$ takes values in $\Lambda_\eta$
and is of degree $(N+1,n-k)$ it vanishes and thus the first term in
the right-most expression in 
\eqref{umgas} equals 
\begin{equation*}
\pm H^0_{p+k+1}\wedge \delta_\eta g_{n-k,n-k}=\pm H^0_{p+k+1}\wedge
\dbar g_{n-k-1,n-k-1}=\pm \dbar \big (H^0_{p+k+1}\wedge
g_{n-k-1,n-k-1} \big ), 
\end{equation*}
where we have used that $\nabla_{\eta} g=0$ and that $H^0_{p+k+1}$ is
holomorphic. 
Using that  $H_{p+k}^1\wedge g_{n-k,n-k}$ and
$H^0_{p+k+1}\wedge g_{n-k-1,n-k-1}$ take values in $\Lambda_\eta$ and
have degree $(N,*)$ we get that 
\begin{equation*}
f^*_{p+k+1}H^0_{p+k}\wedge g_{n-k,n-k}=d\eta \wedge\big (\dbar A + f_1(z)
B\big )
\end{equation*}
for some smooth $A$ and $B$. 
Hence 
\begin{equation}
f_{p+k+1}^*\int_{V_z} G(z,\zeta)\wedge \psi(z)=
\int_{V_z} \dbar A \wedge \psi(z) + \int_{V_z} f_1(z) B \wedge
\psi(z)=0. 
\end{equation}
The first integral vanishes by Stokes' theorem since $\psi$ is
$\dbar$-closed and $G$ has compact support in $z$ since $g$ has. The
second integral vanishes since $f_1(z)=0$ on $V_z$.

\smallskip 

If $V$ is Cohen-Macaulay, then \eqref{eq:dualcplx} is exact except for at level $p$ and so $(\B_V^{n,\bullet},\debar)$
is exact except for at level $0$ where the cohomology is $\omega_V^{n,0}=\textrm{ker}\, (\debar \colon \B_V^{n,0}\to \B_V^{n,1})$.
Thus, \eqref{RSWkomplexCM} is exact.
\end{proof}

\section{The trace map}\label{sec:pair}
The basic result of this section is the following theorem. It is the key to define our trace map. 

\begin{theorem}\label{thm:parning}
Let $X$ be a reduced complex space of pure dimension $n$. There is a unique map
\begin{equation*}
\wedge \colon \B^{n,q}_X\times \A^{0,q'}_X \to   \W^{n,q+q'}_X \cap \textrm{Dom}\, \debar
\end{equation*}
extending the exterior product on $X_{reg}$. 
\end{theorem}

The uniqueness is clear since two currents with the SEP that are equal on $X_{reg}$ are equal on $X$.
It is moreover clear that $\wedge$ is $\mathcal{E}^{0,0}_X$-bilinear. Indeed, if, e.g., $\varphi_1$ and
$\varphi_2$ are sections of $\A_X^{0,q'}$, $\psi$ is a section of $\B_X^{n,q}$, and $\xi_1$ and $\xi_2$
are sections of $\mathcal{E}^{0,0}_X$, then $\psi\wedge (\xi_1\varphi_1+\xi_2\varphi_2)$, $\psi\wedge\xi_1\varphi$,
and $\psi\wedge\xi_2\varphi_2$ have the SEP by Theorem~\ref{thm:parning} and
$\psi\wedge (\xi_1\varphi_1+\xi_2\varphi_2)=\psi\wedge\xi_1\varphi_1 + \psi\wedge\xi_2\varphi_2$ on $X_{reg}$.
We get bilinear pairings of $\C$-vector spaces, $\B_c^{n,n-q}(X)\times \A^{0,q}(X)\to \C$ and 
$\B^{n,n-q}(X)\times \A_c^{0,q}(X)\to \C$, given by $(\psi,\varphi)\mapsto \int_X \psi\wedge\varphi:= \psi\wedge\varphi . 1$,
where $1$ here denotes the function constantly equal to $1$; we will refer to these maps as 
\emph{trace maps on the level of currents}. We also get \emph{trace maps on the level of cohomology}:

\begin{corollary}\label{tracecor}
Let $\varphi$ and $\psi$ be sections of $\A_X^{0,q'}$ and $\B_X^{n,q}$ respectively. Then
$\debar(\psi\wedge\varphi)=\debar\psi\wedge\varphi \pm \psi\wedge\debar\varphi$. Moreover,
there are bilinear maps of $\C$-vector spaces
\begin{equation*}
H^q\big(\A^{0,\bullet}(X), \debar\big) \times H^{n-q}\big(\B_c^{n,\bullet}(X),\debar\big) \to \C,
\end{equation*}
\begin{equation*}
H^q\big(\A_c^{0,\bullet}(X), \debar\big) \times H^{n-q}\big(\B^{n,\bullet}(X),\debar\big) \to \C,
\end{equation*}
given by $([\varphi]_{\debar},[\psi]_{\debar})\mapsto \int_X \psi\wedge \varphi$.
\end{corollary}

\begin{proof}
By Theorem~\ref{thm:parning}, $\debar(\psi\wedge\varphi)$ has the SEP; cf.\ Definition~\ref{def:domain}.
By Theorem~\ref{thm:debar} and \cite[Theorem~1.2]{AS}, respectively, $\debar\psi$ is a section of $\B_X^{n,q+1}$ and 
$\debar\varphi$ is a section of $\A_X^{0,q'+1}$. Thus, $\debar\psi\wedge\varphi$ and $\psi\wedge\debar\varphi$
have the SEP by Theorem~\ref{thm:parning} and so 
$\debar(\psi\wedge\varphi)=\debar\psi\wedge\varphi \pm \psi\wedge\debar\varphi$ since it holds on $X_{reg}$.
The last part of the corollary immediately follows.
\end{proof}

\begin{proof}[Proof of Theorem~\ref{thm:parning}]
We have already noticed that if $\psi|_{X_{reg}}\wedge\varphi|_{X_{reg}}$ has an extension with the SEP, then it is unique.
To see that such an extension exists, let $V$ be a relatively compact open subset of a pure 
$n$-dimensional analytic subset of some pseudoconvex
domain in some $\C^N$.   
Let $\phi=(\phi_1,\ldots,\phi_s)$ be generators for the radical ideal sheaf over $V\times V$
associated to the diagonal $\Delta^V \subset V\times V$. Let
\begin{equation*}
A_{\epsilon}:=\chi(|\phi|/\epsilon) \frac{\partial \log |\phi|^2}{2\pi i} \wedge
(dd^c \log |\phi|^2)^{n-1}.
\end{equation*}
Notice that if $p\colon W\to V\times V$ is a holomorphic map such that, locally on $W$, $p^*\phi=\phi_0\phi'$
for a holomorphic function $\phi_0$ and a non-vanishing holomorphic tuple $\phi'$, then
\begin{equation}\label{pool}
2\pi i p^*A_{\epsilon} = \chi(|\phi_0\phi'|/\epsilon)
\big(d\phi_0/\phi_0 + \partial |\phi'|^2/|\phi'|^2\big)\wedge (dd^c \log |\phi'|^2)^{n-1}.
\end{equation}
Thus, in view of Section~\ref{ssec:PM}, $A:=\lim_{\epsilon\to 0}A_{\epsilon}$ exists and defines an almost semi-meromorphic
current on $V\times V$. Let 
\begin{equation}\label{lut}
M_{\epsilon}:=\debar\chi(|\phi|/\epsilon)\wedge \frac{\partial \log |\phi|^2}{2\pi i} \wedge
(dd^c \log |\phi|^2)^{n-1} = \debar A_{\epsilon} - \chi(|\phi|/\epsilon) (dd^c\log |\phi|^2)^n. 
\end{equation}
Similarly to \eqref{pool} one checks that the limit of the last term on the right-hand side defines an
almost semi-meromorphic current. Thus, the limit $M:=\lim_{\epsilon\to 0}M_{\epsilon}$ exists and defines a pseudomeromorphic
$(n,n)$-current on $V\times V$ supported on $\Delta^V$. Notice that $M$ is the difference of an almost semi-meromorphic current and
the $\debar$-image of such a current. Hence, by Proposition~\ref{multprop}, for any pseudomeromorphic current $\tau$,
$M\wedge\tau$ is a well-defined pseudomeromorphic current.
It is well-known that $M=[\Delta^V]$ on $V_{reg}\times V_{reg}$
and so, in view of the dimension principle, $M=[\Delta^V]$ on $V\times V$; cf.\ \cite[Corollary~1.3]{ASWY}. 

Let $\psi\in \B^{n,q}(V)$ and $\varphi\in\A^{0,q'}(V)$. The tensor product $\psi(w)\wedge\varphi(z)$
is a pseudomeromorphic current on $V\times V$ by Section~\ref{ssec:PM}, and so
$M\wedge\psi(w)\wedge\varphi(z)=\lim_{\epsilon\to 0} M_{\epsilon}\wedge\psi(w)\wedge\varphi(z)$
is a pseudomeromorphic currents on $V\times V$ with support on $\Delta^V$.
Notice also that
since $\psi$ and $\varphi$ are smooth on $V_{reg}$, we have 
\begin{equation}\label{plutten}
M\wedge\psi(w)\wedge\varphi(z)=[\Delta^V]\wedge\psi(w)\wedge\varphi(z)=i_*(\psi|_{V_{reg}}\wedge\varphi|_{V_{reg}})
\end{equation} 
on $V_{reg}\times V_{reg}$,
where $i\colon \Delta^V\to V\times V$ is the inclusion and where we have made the identification $\Delta^V\simeq V$.

\begin{lemma}\label{lma:main}
The pseudomeromorphic currents $M\wedge\psi(w)\wedge\varphi(z)$ and $\debar\big(M\wedge\psi(w)\wedge\varphi(z)\big)$
have the SEP with respect to $\Delta^V$.
\end{lemma}

Let $g$ be a holomorphic function such that $g|_{\Delta^V}=0$. Then $g[\Delta^V]=0=dg\wedge [\Delta^V]$
and so, since $\psi(w)\wedge\varphi(z)$ is smooth on  $V_{reg}\times V_{reg}$ and $M=[\Delta^V]$, we have
\begin{equation}\label{luta}
gM\wedge\psi(w)\wedge\varphi(z)=dg\wedge M\wedge\psi(w)\wedge\varphi(z)=0
\end{equation}
on $V_{reg}\times V_{reg}$.
In fact, by Lemma~\ref{lma:main}, \eqref{luta} holds on $V\times V$ and so, by Proposition~\ref{matssats}
and Lemma~\ref{lma:main} again,
there is a $\mu\in \W(V)$ such that $M\wedge\psi(w)\wedge\varphi(z)=i_*\mu$. Hence,
in view of \eqref{plutten}, $\mu$ is an extension of $\psi|_{V_{reg}}\wedge\varphi|_{V_{reg}}$ to $V$
with the SEP. We will denote the extension by $\psi\wedge\varphi$.

It remains to see that $\psi\wedge\varphi$ is in $\text{Dom}\, \debar$.
However, $\debar\big(M\wedge\psi(w)\wedge\varphi(z)\big)=i_* \debar(\psi\wedge\varphi)$ and 
$\debar\big(M\wedge\psi(w)\wedge\varphi(z)\big)$
has the SEP with respect to $\Delta^V$ by Lemma~\ref{lma:main}. It follows that $\debar(\psi\wedge\varphi)$
has the SEP on $V$, i.e., $\psi\wedge\varphi$ is in $\text{Dom}\, \debar$.
\end{proof}

\begin{proof}[Proof of Lemma~\ref{lma:main}]
We may assume, cf.\ Definition~\ref{A-sheaves} and the end of Section~\ref{ssec:koppel0qX}, that
\begin{equation*}
\psi=\xi_m\wedge \check{\mathscr{K}}_{m} \left(\cdots \xi_1 \wedge \check{\mathscr{K}}_{1}(\omega\wedge \xi_0)\cdots\right),
\quad
\varphi=\tilde{\xi}_{\ell}\wedge \mathscr{K}_{\ell} \left(\cdots \tilde{\xi}_1 \wedge \mathscr{K}_{1}(\tilde{\xi}_0)\cdots\right),
\end{equation*}
where $\xi_i$ and $\tilde{\xi}_j$ are smooth $(0,*)$-forms, $\omega=\sum_k\omega_k$ is a structure form
associated with a free resolution \eqref{eq:resolOJ}, and
$\check{\mathscr{K}}_{i}$ and $\mathscr{K}_{j}$ are integral operators for $(n,*)$-forms and 
$(0,*)$-forms respectively.
Let $\check{k}_j(w^{j-1},w^j)$ be the integral kernel corresponding to $\check{\mathscr{K}}_{j}$
and let $k_j(z^j,z^{j-1})$ be the integral kernel corresponding to $\mathscr{K}_{j}$;
$w^j$ and $z^j$ are coordinates on $V$.
We will assume that for each $j$, $z^{j}\mapsto k_{j+1}(z^{j+1},z^j)$ has compact support where 
$z^{j}\mapsto k_j(z^j,z^{j-1})$ is defined and similarly for $\check{k}_j$; 
possibly we will have to multiply by a smooth cut-off function
that we however will suppress. Now, consider
\begin{equation}\label{hund0}
T:= \lim_{\epsilon\to 0} M_{\epsilon}(z^{\ell},w^m)\wedge
\bigwedge_{j=1}^m \check{k}_j(w^{j-1},w^j) \wedge \omega(w^0) \wedge \bigwedge_{j=1}^{\ell} k_j(z^j,z^{j-1}),
\end{equation}
which is a pseudomeromorphic current on $V^{\ell+m+2}$ supported on $\{z^{\ell}=w^m\}$; 
cf.\ Proposition~\ref{multprop}.\footnote{In this proof 
$V^j$ will mean either the Cartesian product of $j$ copies of $V$ or the $j^{\textrm{th}}$ 
set in \eqref{ballong}. We hope that it will be clear from the context what we are aiming at.}
Notice that $M(z^{\ell},w^m)\wedge\psi(w^m)\wedge\varphi(z^{\ell})=\pi_* (T\wedge\xi)$, where 
$\pi\colon V^{\ell+m+2}\to V_{z^{\ell}}\times V_{w^m}$ is the natural projection and $\xi$ is a suitable 
smooth form on $V^{\ell+m+2}$. In view of the paragraph following the dimension principle in Section~\ref{ssec:PM},
 it suffices to show that $T$ and $\debar T$ have the SEP with respect to $\{z^{\ell}=w^m\}$.
Let $h=h(z^{\ell},w^m)$ be a germ of a holomorphic tuple in $V\times V$ that is generically non-vanishing on the diagonal;
we will consider $h$ also as a germ of a tuple on $V^{\ell+m+2}$ and we denote its zero-set there by $H$.
In view of Section~\ref{ssec:PM}, what we are to show is that  
$\mathbf{1}_H T=\mathbf{1}_H \debar T=0$.

Let $T_k$ be the part of $T$ corresponding to $\omega_k(w^0)$ and notice that $T_k$ is a pseudomeromorphic
current of bidegree $(*,n(\ell+m+1)-m-\ell+k)$. We will show that $T$ and $\debar T$ have the SEP
by double induction over $\ell+m$ and $k$.

Assume first that $\ell=m=0$. Then $T_k=M(z^0,w^0)\wedge\omega_k(w^0)$ and we know that $T_k=[\Delta^V]\wedge\omega_k(w^0)$
for $w^0\in V_{reg}$ since $\omega_k(w^0)$ is smooth there. Hence, since $[\Delta^V]$ has the SEP with respect to $\Delta^V$,
$\mathbf{1}_H T_k=0$ outside of $\{w^0\in V_{sing}\}$
and it follows that $\text{supp} (\mathbf{1}_H T_k)\subset \{z^0=w^0\in V_{sing}\}$, which has codimension $\geq n+1$ in
$V\times V$. Since $\mathbf{1}_H T_0$ has bidegree $(*,n)$, the dimension principle implies that $\mathbf{1}_H T_0=0$.
By Proposition~\ref{structureprop}, $\omega_k=\alpha_k\omega_{k-1}$, where $\alpha_k$ is smooth 
outside of $V^k$, which has codimension $\geq k+1$ in $V$. 
Hence, $\text{supp}\,\mathbf{1}_H T_1\subset\{w^0\in V^1\}$, which has codimension $\geq n+2$ in $V\times V$.
Since $\mathbf{1}_H T_1$ has bidegree $(*,n+1)$, the dimension principle implies that also $\mathbf{1}_H T_1=0$.
Continuing in this way, we get that $\mathbf{1}_H T_k=0$. Hence, $T=M(z^0,w^0)\wedge\omega(w^0)$ has the SEP with respect to 
$\Delta^V$ and arguing as in the paragraph following Lemma~\ref{lma:main} we see that $T=i_*\omega$.
Since $\debar\omega=f\omega$ by Proposition~\ref{structureprop}, it follows that
$\debar T=i_*\debar\omega=i_*f\omega$ and thus, $\debar T$ has the SEP with respect to $\Delta^V$.

Let now $\ell+m=s\geq 1$ in \eqref{hund0} and assume that $T$ and $\debar T$ have the SEP with respect to $\{z^{\ell}=w^m\}$ for
$\ell+m\leq s-1$. Let $1\leq r\leq \ell$;
if $z^{r-1}\neq z^{r}$ then
$k_r(z^r,z^{r-1})$ is a smooth form times some structure form $\tilde{\omega}(z^{r-1})$. Hence,
outside of $\{z^r=z^{r-1}\}$,
$T$ is a smooth form times the {\em tensor product} of
\begin{equation*}
\tilde{\omega}(z^{r-1}) \bigwedge_{j=1}^{r-1} k_j(z^j,z^{j-1})
\end{equation*}
and some current $\tilde{T}$, where $\tilde{T}$ is of the form \eqref{hund0} with $\ell+m=s-r$
depending on the variables $z^r,\ldots,z^{\ell}$ and $w^0,\ldots,w^m$.
From the induction hypothesis it thus follows that $\mathbf{1}_H T$ and 
$\mathbf{1}_H\debar T$ have supports contained in 
$\{z^0=\ldots=z^{\ell}\}$.
Similarly, let $1\leq r\leq m$. If $w^{r-1}\neq w^r$ then $\check{k}_r(w^{r-1},w^r)$ is a smooth form times
some structure form $\tilde{\omega}(w^r)$ and so, outside of $\{w^{r-1}=w^{r}\}$, $T$ is a smooth form times
the tensor product of 
\begin{equation*}
\bigwedge_{j=1}^{r-1} \check{k}_j(w^{j-1},w^j)\wedge \omega(w^0)
\end{equation*} 
and a current of the form \eqref{hund0} with $\ell+m=s-r$ depending on the variables $z^0,\ldots,z^{\ell}$ and $w^r,\ldots,w^m$. 
Thus, again from the induction hypothesis,
it follows that $\mathbf{1}_H T$ and $\mathbf{1}_H\debar T$ have supports contained in
$\{w^0=\ldots=w^{m}\}$. In addition, since $T$ vanishes outside of $\{z^{\ell}=w^m\}$,
we have that the supports of $\mathbf{1}_H T$ and $\mathbf{1}_H\debar  T$
must be contained in the diagonal $\Delta^V=\{z^0=\cdots=z^{\ell}=w^m=\cdots=w^0\}\subset V^{\ell+m+2}$.
Hence, we see that $\mathbf{1}_H T$ and $\mathbf{1}_H\debar T$ have supports contained
in $\Delta^V\cap H$, which has codimension $\geq n(s+1)+1$. Since $\mathbf{1}_H T_0$ has bidegree $(*,n(s+1)-s)$ and 
$\mathbf{1}_H\debar T_0$ has bidegree $(*,n(s+1)-s+1)$ we have 
$\mathbf{1}_H T_0=\mathbf{1}_H\debar T_0=0$ by the dimension principle. Since $T_1=\pm \alpha_1(w^0)T_0$ and $\alpha_1$ is smooth 
outside of $V^1$, which has codimension $\geq 2$ in $V$, it follows that $\mathbf{1}_H T_1$ and $\mathbf{1}_H\debar T_1$
have supports in $\Delta^V\cap\{w^0\in V^1\}$, which then has codimension $\geq n(s+1)+2$. The dimension principle then
shows that $\mathbf{1}_H T_1=\mathbf{1}_H\debar T_1=0$. By induction over $k$, using that $T_k=\pm \alpha_k(w^0)T_{k-1}$
with $\alpha_k$ smooth outside of $V^k$, that $\text{codim}_V\, V^k\geq k+1$, and the dimension principle, we obtain 
$\mathbf{1}_H T_k=\mathbf{1}_H\debar T_k=0$ for all $k$. 
\end{proof}

\section{Serre duality}

\subsection{Local duality}\label{ssec:locdual1}
Let $V$ be a pure $n$-dimensional analytic subset of a pseudoconvex domain $D\subset \C^N$, let 
$D'\Subset D$ be a strictly pseudoconvex subdomain, and let $V'=V\cap D'$.  
Consider the complexes
\begin{equation}\label{eq:cplx1}
0\to \A^{0,0}(V') \stackrel{\debar}{\longrightarrow} \A^{0,1}(V') 
\stackrel{\debar}{\longrightarrow} \cdots \stackrel{\debar}{\longrightarrow} \A^{0,n}(V') \to 0
\end{equation}
\begin{equation}\label{eq:cplx2}
0 \to \B^{n,0}_{c}(V') \stackrel{\debar}{\longrightarrow}
\B^{n,1}_{c}(V') \stackrel{\debar}{\longrightarrow} \cdots
\stackrel{\debar}{\longrightarrow} \B^{n,n}_{c}(V') \to 0.
\end{equation}
From Corollary~\ref{tracecor} we have the trace map
\begin{equation}\label{eq:pair1}
Tr\colon H^0\left(\A^{0,\bullet}(V'),\debar\right) \times H^n\left(\B_{c}^{n,\bullet}(V'),\debar\right) \to \C, \quad 
Tr([\varphi],[\psi]) = \int_{V'} \varphi \psi.
\end{equation}
By \cite[Theorem~1.2]{AS}
the complex \eqref{eq:cplx1} is exact except for at the level $0$ where the cohomology
is $\hol(V')$, cf.\ the introduction. 

\begin{theorem}\label{thm:locdual1}
The complex \eqref{eq:cplx2} is exact except for at the top level and
the pairing \eqref{eq:pair1} makes $H^n(\B^{n,\bullet}_{c}(V'))$ the topological dual of the 
Frech\'{e}t space $H^0(\A^{0,\bullet}(V'))=\hol(V')$; in particular \eqref{eq:pair1} is non-degenerate.
\end{theorem}

\begin{proof}
Let $\psi\in \B^{n,q}_{c}(V')$ be $\debar$-closed. Moreover, let $g$
be a weight 
in $D''\times D'$, where $D''\subset D'$ is a neighborhood of
$\textrm{supp}\, \psi$, such that $g$ is holomorphic in $z$ and has
compact support in $D'_\zeta$, cf.\ Section~\ref{ssec:koppelCn}, and let $k(z,\zeta)$ and $p(z,\zeta)$
be the integral kernels defined by \eqref{eq:k} and \eqref{eq:p},
respectively. 
Since $\psi$ has compact support in $D''$, Theorem~\ref{thm:Homotopy} shows that
\begin{equation}\label{krut}
\psi(\zeta) = \debar_{\zeta}\int_{V'_z} k(z,\zeta)\wedge \psi(z) + 
\int_{V'_z} k(z,\zeta)\wedge \debar\psi(z) + \int_{V'_z} p(z,\zeta)\wedge \psi(z),
\end{equation}
holds on $V'$. The second term on the right hand side vanishes since $\debar \psi=0$. 
Since $g$ is holomorphic in $z$ the kernel $p$ has degree $0$ in $d\bar{z}_j$ and hence, also
the last term vanishes if $q\neq n$.
The first integral on the right hand side is in $\B^{n,q-1}_{c}(V')$ since $g$ has compact support in 
$D'_{\zeta}$ and so \eqref{eq:cplx2} is exact except for at level $n$.

To see that $H^n(\B^{n,\bullet}_{c}(V'))$ is the topological dual of $\hol(V')$,
recall that the topology on $\hol(V')\cong \hol(D')/\mathcal{J}(D')$ is the quotient topology,
where $\J_V$ be the sheaf in $D$ associated with $V\subset D$.
It is clear that each $[\psi]\in H^n(\B^{n,\bullet}_{c}(V'))$ yields a continuous linear functional on $\hol(V')$
via \eqref{eq:pair1}. Moreover, if $q=n$ and $\int_{V'} \varphi\psi=0$ for all $\varphi\in \hol(V')$ then, 
since $p(z,\zeta)$ is holomorphic in $z$ by the choice of $g$, the last integral on the right hand side of \eqref{krut}
vanishes and thus $[\psi]=0$. Hence, $H^n(\B^{n,\bullet}_{c}(V'))$ is a subset of the topological dual of $\hol(V')$.

To see that there is equality, let $\lambda$ be a continuous linear functional on $\hol(V')$. By composing with the projection
$\hol(D') \to \hol(D')/\mathcal{J}(D')$ we get a continuous functional $\tilde{\lambda}$ on $\hol(D')$.
By definition of the topology on $\hol(D')$, $\tilde{\lambda}$ is carried by some compact subset $K\Subset D'$.
By the Hahn-Banach theorem, $\tilde{\lambda}$ can be extended to a 
continuous linear functional on $C^0(D')$ and so it is given as integration against some measure $\mu$ on $D'$
that has support in a neighborhood $U(K)\Subset D'$ of $K$.
Let $\tilde g$ be a weight in $U(K)\times D'$ 
that 
depends holomorphically on $z\in U(K)$ and that has compact support in
$D'_\zeta$, and let $\tilde{p}(z,\zeta)$ be the integral kernel
defined from $\tilde g$ as in \eqref{eq:p}, and let $\mathscr P$ be
the corresponding integral operator.  
Let $f\in \hol(V')$ and define the sequence $f_{\epsilon}(z)\in \hol(K)$ by
\begin{equation*}
f_{\epsilon}(z) = \int_{V'_{\zeta}} \chi_\epsilon(\zeta)
\tilde{p}(z,\zeta) f(\zeta), 
\end{equation*}
where, as above, $\chi_\epsilon=\chi(|h|/\epsilon)$ and $h=h(\zeta)$ is a holomorphic tuple such that
$\{h=0\}=V_{sing}$.
For each $z$ in a neighborhood in $V'$ of $K\cap V'$ we have that
$\lim f_{\epsilon}(z)=\mathscr{P} f (z)=f(z)$  
by \cite[Theorem~1.4]{AS}. We claim that $f_{\epsilon}$ in fact converges uniformly 
in a neighborhood of $K$ in $D'$ to some $\tilde{f}\in \hol(K)$, which then is an extension of $f$ to a neighborhood 
in $D'$ of $K$. 
To see this, first notice by \eqref{eq:p} that 
$\tilde{p}(z,\zeta)$ is a sum of terms $\omega_k(\zeta)\wedge p_k(z,\zeta)$ where
$p_k(z,\zeta)$ is smooth in both variables and holomorphic for $z\in U(K)$. 
By Proposition~\ref{structureprop}, the $\omega_k$ are almost semi-meromorphic. The claim then follows from a simple instance
of \cite[Theorem 1]{JEBHS}\footnote{Take $p=0$, $q=1$, and $\mu=1$ in this theorem.}. 
We now get
\begin{eqnarray*}
\lambda(f) &=& \lim_{\epsilon\to 0}\int_z f_{\epsilon}(z) d\mu(z) = 
\lim_{\epsilon\to 0}\int_{z} \int_{V'_{\zeta}} \chi_\epsilon(\zeta)
\tilde{p}(z,\zeta) f(\zeta) d\mu(z) \\
&=&
\lim_{\epsilon\to 0}\int_{V'_{\zeta}} f(\zeta) \chi_\epsilon(\zeta)
\int_{z} \tilde{p}(z,\zeta) d\mu(z) \\
&=& 
\lim_{\epsilon\to 0}\int_{V'_{\zeta}} f(\zeta) \chi_\epsilon(\zeta)
\sum_k \omega_k(\zeta)\wedge \int_{z} p_k(z,\zeta) d\mu(z) \\
&=&
\int_{V'_{\zeta}} f(\zeta) \sum_k \omega_k(\zeta)\wedge \int_{z} p_k(z,\zeta) d\mu(z).
\end{eqnarray*}
But $\zeta \mapsto \int_{V_z} p_k(z,\zeta) d\mu(z)$ is smooth and compactly supported in $D'$ 
and so $\lambda$ is given as integration against
some element $\psi \in \B_{c}^{n,n}(V')$; hence $\lambda$ is realized by the cohomology class $[\psi]$ and the
theorem follows.
\end{proof}

\begin{corollary}\label{rmk:locdual1}
Let $F\to V$ be a vector bundle, $\F=\hol(F)$ the associated locally free $\hol$-module, and $\F^*=\hol(F^*)$.
Then the following pairing is non-degenerate
\begin{equation*}
Tr\colon H^0(V', \F) \times H^n\big (\F^*\otimes
\B^{n,\bullet}_{c}(V')\big ) \to \C,  \quad ([\varphi],[\psi]) \mapsto \int_{V'}\varphi \psi.
\end{equation*}
\end{corollary}

By Theorem~\ref{thm:resol}, if $X$ is Cohen-Macaulay, 
then the complex $(\F^*\otimes \B^{n,\bullet}_V, \debar)$ is a resolution of $\F^*\otimes \omega^{n,0}_V$
and so we get a non-degenerate pairing
\begin{equation*}
H^0(V', \F) \times H^n_{c}(V', \F^*\otimes \omega^{n,0}_{V}) \to \C.
\end{equation*}

\subsection{Global duality}\label{ssec:globdual1}
From the local duality an abstract global duality follows by a patching argument using
\v{C}ech cohomology, see \cite{RaRu}, cf.\ also \cite[Theorem~(I)]{AK}. 
To see that this abstract global duality is realized by Theorem~\ref{thm:main2}
we will make this patching argument explicit using a perhaps
non-standard formalism for \v{C}ech cohomology; cf.\ \cite[Section~7.3]{Horm}

Let $\F$ be a sheaf on $X$ and let $\V=\{V_j\}$ be a locally finite covering of $X$.
We let $C^k(\V, \F)$ be the group of formal sums
\begin{equation*}
\sum_{i_0 \cdots i_k} f_{i_0\cdots i_k} V_{i_0}\wedge \cdots \wedge V_{i_k}, \quad  f_{i_0\cdots i_k}
\in \F(V_{i_0}\cap \cdots \cap V_{i_k})
\end{equation*}
with the suggestive computation rules, e.g., $f_{12}V_1\wedge V_2 + f_{21}V_2\wedge V_1= (f_{12}-f_{21})V_1\wedge V_2$.
Each element of $C^k(\V, \F)$ thus has a unique representation of the form
\begin{equation*}
\sum_{i_0 < \cdots < i_k} f_{i_0\cdots i_k} V_{i_0}\wedge \cdots \wedge V_{i_k}
\end{equation*}
that we will abbreviate as $\sum'_{|I|=k+1} f_IV_I$.
The coboundary operator $\delta \colon C^k(\V, \F) \to C^{k+1}(\V, \F)$ can in this formalism
be taken to be the formal wedge product 
\begin{equation*}
\delta (\sum'_{|I|=k+1} f_I V_I) = (\sum'_{|I|=k+1} f_I V_I) \wedge (\sum_j V_j).
\end{equation*}
If $\V$ is a Leray covering for $\F$, then $H^k(C^{\bullet}(\V, \F), \delta)\cong H^k(X, \F)$.
Indeed, let $(\F^{\bullet},d)$ be a flabby resolution of $\F$. Then $H^k(X,\F)=H^k(\F^{\bullet}(X),d)$ and
applying standard homological algebra to the double complex
$C^{\bullet}(\V, \F^{\bullet})$ one shows that $H^k(C^{\bullet}(\V,\F),\delta)\simeq H^k(\F^{\bullet}(X), d)$. 
If $\F$ is fine, i.e., a $\mathcal{E}^{0,0}_X$-module, then the complex $(C^{\bullet}(\V, \F), \delta)$
is exact except for at level $0$ where $H^0(C^{\bullet}(\V, \F), \delta)\cong H^0(X, \F)$.

\smallskip

Let $\G'$ be a precosheaf on $X$. Recall, see, e.g., \cite[Section~3]{AK}, 
that a precosheaf of abelian groups is an 
assignment that to each open set $V$ associates an abelian group $\G'(V)$, together with
inclusion maps $i^V_W\colon \G'(V)\to \G'(W)$ for $V\subset W$ such that 
$i^{V'}_W=i^{V}_W i^{V'}_V$ if $V'\subset V\subset W$. 
We define $C^{-k}_c(\V, \G')$ to be the group of formal sums
\begin{equation*}
\sum_{i_0 \cdots i_k} g_{i_0\cdots i_k} V^*_{i_0}\wedge \cdots \wedge V^*_{i_k}, 
\end{equation*}
where $g_{i_0\cdots i_k} \in \G'(V_{i_0}\cap \cdots \cap V_{i_k})$ 
and only finitely many $g_{i_0\cdots i_k}$ are non-zero;
for $k<0$ we let $C_c^{-k}(\V,\G')=0$.
We define a coboundary operator $\delta^*\colon C^{-k}_c(\V, \G') \to C^{-k+1}_c(\V, \G')$ by
formal contraction
\begin{equation*}
\delta^*(\sum_{|I|=k+1}' g_I V^*_I) = \sum_j V_j \lrcorner \sum_{|I|=k+1}' g_I V^*_I,
\end{equation*}
see \eqref{aj} and \eqref{oj} below.
If $\G$ is a sheaf (of abelian groups), then 
$V\to \G_{c}(V)$ is a precosheaf $\G'$ by extending sections by $0$.
We will write $C_c^{-k}(\V,\G)$ in place of $C_c^{-k}(\V,\G')$.

Assume now that there, for every open $V\subset X$, is a map $\F(V)\otimes \G'(V) \to \F'(V)$ where $\F'$ and $\G'$
are precosheaves on $X$. We then define a contraction map 
$\lrcorner\colon C^k(\V, \F) \times C^{-\ell}_c(\V, \G') \to C_c^{k-\ell}(\V,\F')$ by using the following computation rules. 
\begin{equation}\label{aj}
V_i \lrcorner V_j^* =
\left\{ \begin{array}{cc}
1, & i=j \\
0, & i\neq j \end{array}\right. ,
\end{equation}
\begin{equation}\label{oj}
V_i\lrcorner (V^*_{j_0}\wedge \cdots \wedge V^*_{j_{\ell}}) = \sum_{m=0}^{\ell} (-1)^{m}
V^*_{j_0}\wedge \cdots (V_i \lrcorner V^*_{j_m}) \cdots \wedge V^*_{j_{\ell}},
\end{equation}
\begin{equation*}
(V_{i_0}\wedge \cdots \wedge V_{i_k})\lrcorner V_J^* = 
\left\{ \begin{array}{cc}
0, & k > |J| \\
((V_{i_0}\wedge \cdots \wedge V_{i_{k-1}}))\lrcorner (V_{i_k}\lrcorner V_J^*), & k\leq |J| \end{array}\right. .
\end{equation*}
If $\F'$ and $\G'$ are sheaves we define in a similar way also the contraction
$\lrcorner\colon C^{-k}_c(\V, \G') \times C^{\ell}(\V, \F) \to C^{\ell-k}(\V,\F')$.
If $g=g_I V_I^*$ and $f=f_J V_J$, then
$g\lrcorner f=g_I f_J V_I^* \lrcorner V_J$, where $g_I f_J$
is the extension to $\bigcap_{i\in J\setminus I}V_i$ by $0$; this is
well-defined since $g_I f_J$ is $0$ in a neighborhood of the boundary
of $\bigcap_{j\in J}V_j$ in $\bigcap_{i\in J\setminus I}V_i$.

\begin{lemma}\label{lma:finecech}
If $\G$ is a fine sheaf, then 
\begin{equation*}
H^{-k}(C^{\bullet}_c(\V, \G), \delta^*) = \left\{
\begin{array}{cc}
0, & k \neq 0 \\
H^0_{c}(X, \G), & k=0
\end{array}\right. . 
\end{equation*}
\end{lemma}

\begin{proof}
Let $\{\chi_j\}$ be a smooth partition of unity subordinate to $\V$ and let
$\chi=\sum_j \chi_j V^*_j$. Since $\delta^* \chi =\sum \chi_j=1$ we have 
\begin{equation*}
\delta^*(\chi\wedge g) = \delta^*(\chi) \cdot g - \chi\wedge \delta^*(g) = g - \chi\wedge \delta^*(g)
\end{equation*}
for $g\in C^{-k}_c(\V, \G)$.
Hence, if $g$ is $\delta^*$-closed, then $g$ is $\delta^*$-exact. It follows that the complex
\begin{equation*}
\cdots \stackrel{\delta^*}{\longrightarrow} C^{-1}_c(\V, \G) \stackrel{\delta^*}{\longrightarrow} 
C^{0}_c(\V, \G) \stackrel{\delta^*}{\longrightarrow}
H^0_{c}(X, \G) \to 0
\end{equation*}
is exact and so the lemma follows.
\end{proof}

\smallskip

\begin{center}
---
\end{center}

\smallskip

Let $X$ be a paracompact reduced complex space of pure dimension $n$.
Let $\aleph$ be the precosheaf on $X$ defined by 
\begin{equation*}
\aleph(V)=H^n(\B^{n,\bullet}_{c}(V), \debar),
\end{equation*}
\begin{equation*}
i^V_W\colon \aleph(V)\to \aleph(W), \quad i^V_W([\psi])=[\tilde{\psi}],
\end{equation*} 
where $\psi\in \B^{n,n}_{c}(V)$ and $\tilde{\psi}$ is the extension of $\psi$ by $0$.\footnote{In view of 
Theorem~\ref{thm:locdual1} and \cite[Proposition~8 (a)]{AK}, $\aleph$ is in fact a cosheaf.}
Let $\V=\{V_j\}$ be a suitable locally finite Leray covering of $X$ and consider the complexes
\begin{equation}\label{eq:cech0}
0  \to  C^0(\V, \hol_X)  \stackrel{\delta}{\longrightarrow}  C^1(\V, \hol_X)  \stackrel{\delta}{\longrightarrow}  \cdots \quad
\end{equation}
\begin{equation}\label{eq:cech}
 \quad \cdots  \stackrel{\delta^*}{\longrightarrow}  C^{-1}_c(\V, \aleph)  
\stackrel{\delta^*}{\longrightarrow}  C^0_c(\V, \aleph)  \to  0.
\end{equation}
By Theorem~\ref{thm:locdual1} we have non-degenerate pairings 
\begin{equation*}
Tr\colon C^k(\V, \hol_X)\times C^{-k}_c(\V, \aleph)\to \C, \quad Tr(f,g)= \int_X f \lrcorner g,
\end{equation*} 
induced by the trace map \eqref{eq:pair1}; 
in fact, Theorem~\ref{thm:locdual1} shows that these pairings make the complex \eqref{eq:cech} the topological dual of 
the complex of Frech\'{e}t spaces \eqref{eq:cech0}.
Moreover, if $f\in C^{k-1}(\V, \hol_X)$ and $g\in C^{-k}_c(\V, \aleph)$ we have
\begin{eqnarray}\label{ost}
Tr(\delta f, g) &=& \int_X (\delta f) \lrcorner g= \int_X \big(f\wedge \sum_j V_j \big) \lrcorner g =
\int_X f \lrcorner \big( (\sum_j V_j)\lrcorner g \big) \\ 
&=& \int_X f \lrcorner (\delta^* g) = Tr(f, \delta^* g). \nonumber
\end{eqnarray}
Hence, we get a well-defined pairing on cohomology level
\begin{equation}\label{eq:glass}
Tr\colon H^k\left( C^{\bullet}(\V, \hol_X) \right) \times H^{-k}\left( C^{\bullet}_c(\V, \aleph) \right) \to \C, \quad
Tr([f],[g]) = \int_X f\lrcorner g.
\end{equation}
Since $\V$ is a Leray covering we have 
\begin{equation}\label{eq:korv}
H^k\left( C^{\bullet}(\V, \hol_X) \right) \cong H^k(X,\hol_X)
\cong H^k\left(\A^{0,\bullet}(X)\right),
\end{equation}
and these isomorphisms induce canonical topologies
on $H^k(X,\hol_X)$ and $H^k\left(\A^{0,\bullet}(X)\right)$; cf.\ \cite[Lemma~1]{RaRu}.
To understand $H^{-k}\left( C^{\bullet}_c(\V,\aleph) \right)$, consider the double complex
\begin{equation*}
K^{-i,j}:=C^{-i}_c(\V, \B^{n,j}_X),
\end{equation*}
where the map $K^{-i,j}\to K^{-i+1, j}$ is the coboundary operator $\delta^*$ and 
the map $K^{-i,j}\to K^{-i,j+1}$ is $\debar$. We have that $K^{-i,j}=0$ if $i<0$ or $j<0$ or $j>n$.
Moreover, the ``rows'' $K^{-i,\bullet}$ are, by Theorem~\ref{thm:locdual1}, exact except for at 
the $n^{\textrm{th}}$ level where the cohomology is $C_c^{-i}(\V,\aleph)$; 
the ``columns'' $K^{\bullet,j}$ are exact except for at level $0$ where the cohomology 
is $\B^{n,j}_{c}(X)$ by Lemma~\ref{lma:finecech} since the sheaf $\B^{n,j}_X$ is fine.
By standard homological algebra (e.g., a spectral sequence argument) it follows that

\begin{equation}\label{eq:ut}
H^{-k}\left(C_c^{\bullet}(\V,\aleph)\right) \cong H^{n-k}\left(\B^{n,\bullet}_{c}(X), \debar\right),
\end{equation}
cf.\ also the proof of Theorem~\ref{thm:main2} below.
The vector space $C_c^{-k}(\V,\aleph)$ has a natural topology since it is the topological dual of 
the Frech\'{e}t space $C^k(\V,\hol_X)$;
therefore \eqref{eq:ut} gives a natural topology on $H^{n-k}(\B^{n,\bullet}_{c}(X))$.

\begin{lemma}\label{lma:hausdorff}
Assume that $H^k(X,\hol_X)$ and $H^{k+1}(X,\hol_X)$, considered as topological vector spaces, 
are Hausdorff. Then the pairing \eqref{eq:glass} is non-degenerate.
\end{lemma}

\begin{proof}
Since \eqref{eq:cech} is the topological dual of \eqref{eq:cech0} it follows (see, e.g., \cite[Lemma~2]{RaRu})
that the topological dual of 
\begin{equation}\label{eq:cohomgroup}
\textrm{Ker} \big(\delta \colon C^k(\V,\hol_X) \to C^{k+1}(\V,\hol_X)\big)/
\overline{\textrm{Im} \big(\delta \colon C^{k-1}(\V,\hol_X) \to C^{k}(\V,\hol_X)\big)}
\end{equation}
equals
\begin{equation}\label{eq:homgroup}
\textrm{Ker} \big(\delta^* \colon C^{-k}_c(\V, \omega^{n,n}_X) \to C^{-k+1}_c(\V,\omega^{n,n}_X)\big)/
\overline{\textrm{Im} \big(\delta^* \colon C^{-k-1}_c(\V,\omega^{n,n}_X) \to C^{-k}_c(\V,\omega^{n,n}_X)\big)}.
\end{equation}
Since $H^k(X,\hol_X)$ and $H^{k+1}(X,\hol_X)$ are Hausdorff it follows that the images of $\delta \colon C^{k-1} \to C^k$ 
and $\delta \colon C^k\to C^{k+1}$ are closed. Since the image of the latter map is closed it follows from 
the open mapping theorem and the Hahn-Banach theorem that also the image of $\delta^*\colon C_c^{-k-1}\to C_c^{-k}$
is closed. The images of $\delta$ and $\delta^*$ in \eqref{eq:cohomgroup}
and \eqref{eq:homgroup} are thus closed and so the closure signs may be removed. Hence, \eqref{eq:glass} makes
$H^{-k}(C_c^{\bullet}(\V,\omega_X^{n,n}))$ the topological dual of $H^k(X,\hol_X)$.
\end{proof}

\begin{remark}
If $X$ is compact the Cartan-Serre theorem says that the cohomology of coherent sheaves on $X$
is finite dimensional, in particular Hausdorff. 
In the compact case the pairing \eqref{eq:glass} is thus always
non-degenerate. 
The pairing \eqref{eq:glass} is also always non-degenerate if
$X$ is holomorphically convex since then, by \cite[Lemma~II.1]{Prill}, $H^k(X,\mathscr{S})$ is
Hausdorff for any coherent sheaf $\mathscr{S}$.

If $X$ is $q$-convex it follows from the Andreotti-Grauert theorem that for any coherent sheaf $\mathscr{S}$, 
$H^k(X,\mathscr{S})$ is Hausdorff for $k\geq q$. Hence, in this case, \eqref{eq:glass} is non-degenerate for $k\geq q$.
\end{remark}

\begin{proof}[Proof of Theorem~\ref{thm:main2}]
For notational convenience we assume that $\F=\hol_X$.
By Lemma~\ref{lma:hausdorff} we know that \eqref{eq:glass} is non-degenerate.
In view of the Dolbeault isomorphisms \eqref{eq:korv} and \eqref{eq:ut} we get an 
induced non-degenerate pairing
\begin{equation*}
Tr \colon H^k\left(\A^{0,\bullet}(X)\right) \times H^{n-k}\left(\B^{n,\bullet}_{c}(X)\right) \to \C.
\end{equation*}
It remains to see that this induced trace map is realized by
$([\varphi],[\psi]) \mapsto \int_X \varphi\wedge \psi$;
for this we will make \eqref{eq:korv} and \eqref{eq:ut} explicit.

Let $\{\chi_j\}$ be a partition of unity subordinate to $\V$, and let
$\chi=\sum_j \chi_j V_j^*$.
We will use the convention that forms \emph{commute} with all $V^*_i$ and $V_j$, i.e.,
if $\xi$ is a differential form then
\begin{equation*}
\xi V^*_I = V^*_I \xi, \quad V^*_I \lrcorner (\xi V_J) = \xi V^*_I\lrcorner V_J.
\end{equation*}
Moreover, we let $\debar(\xi V^*_I)=\debar \xi V^*_I$.
We now let
\begin{equation*}
T_{k,j} \colon C^k(\V,\hol_X) \to C^{k-j-1}(\V,\A^{0,j}_X), \quad T_{k,j}(f) = (\chi\wedge (\debar\chi)^j)\lrcorner f,
\end{equation*}
where we put $C^{-1}(\V, \A_X^{0,k})=\A^{0,k}(X)$ and $C^{\ell}(\V, \A_X^{0,k})=0$ 
for $\ell<-1$.\footnote{In fact, the image of $T_{k,j}$ is contained in $C^{k-j-1}(\V,\mathcal{E}^{0,j}_X)$.}
Using that $\chi \lrcorner V =1$ it is straightforward to verify that
\begin{equation}\label{vad}
T_{k,j}(\delta \tilde{f}) = \delta T_{k-1,j}(\tilde{f}) + (-1)^{k-j}
\debar T_{k-1,j-1} (\tilde{f}), \quad \tilde{f}\in C^{k-1}(\V,\hol_X).
\end{equation}
It follows that if $f\in C^k(\V,\hol_X)$ is $\delta$-closed then $T_{k,k}(f)$ is $\debar$-closed and if 
$f$ is $\delta$-exact then $T_{k,k}(f)$ is $\debar$-exact. Thus $T_{k,k}$ induces a map
\begin{equation*}
\textrm{Dol} \colon H^k(C^{\bullet}(\V,\hol_X)) \to H^k(\A^{0,\bullet}(X)), \quad
\textrm{Dol}([f]_{\delta}) = [T_{k,k}(f)]_{\debar};
\end{equation*}
this is a realization of the composed isomorphism \eqref{eq:korv}.

To make \eqref{eq:ut} explicit, let $[g]\in C^{-k}_c(\V,\aleph)$, where 
$g\in C^{-k}_c(\V,\B^{n,n}_X)$, be $\delta^*$-closed. This means
that there is a $\tau^{n-1}\in C^{-k+1}_c(\V,\B^{n,n-1}_X)$ such that $\delta^*g=\debar \tau^{n-1}$.
Hence, $\debar \delta^* \tau^{n-1} = \delta^* \debar \tau^{n-1} = \delta^*\delta^* g = 0$
and so by Theorem~\ref{thm:locdual1} there is a $\tau^{n-2}\in C_c^{-k+2}(\V,\B_X^{n,n-2})$ such that
$\delta^*\tau^{n-1} = \debar \tau^{n-2}$. Continuing in this way we obtain, for all $j$, 
$\tau^{n-j}\in C_c^{-k+j}(\V,\B_X^{n,n-j})$ such that $\delta^*\tau^{n-j} = \debar \tau^{n-j-1}$.
It follows that $\delta^*\tau^{n-k}\in \B^{n,n-k}_{c}(X)$, cf.\ the proof of Lemma~\ref{lma:finecech},
and that it is $\debar$-closed.
One can verify that if $[g]\in C^{-k}_c(\V,\aleph)$ is $\delta^*$-exact then $\delta^*\tau^{n-k}$ is $\debar$-exact
and so we get a well-defined map
\begin{equation*}
\textrm{Dol}^*\colon H^{-k}(C_c^{\bullet}(\V,\aleph)) \to H^{n-k}(\B^{n,\bullet}_{c}(X)), \quad
\textrm{Dol}^* ([g]_{\debar}) = [\delta^*\tau^{n-k}]_{\debar};
\end{equation*}
this is a realization of the isomorphism \eqref{eq:ut}.

Let now $f\in C^k(\V,\hol_X)$ be $\delta$-closed
and let $[g]\in C^{-k}_c(\V,\aleph)$ be $\delta^*$-closed. 
One checks that $\delta T_{k,0}(f)=(-1)^k f$ and thus, by
\eqref{vad}, we have
\begin{equation*}
\delta T_{k,j}(f) =
\left\{\begin{array}{ll}
(-1)^{k-j} \debar T_{k,j-1}(f), & 1\leq j \leq k \\
(-1)^k f, & j=0
\end{array}.\right.
\end{equation*}
Using this and the computation in \eqref{ost} we get
\begin{eqnarray*}
\int_X f\lrcorner g &=&
(-1)^k\int_X \delta T_{k,0}(f) \lrcorner g = (-1)^k\int_X T_{k,0}(f)\lrcorner \delta^*g 
= (-1)^k\int_X T_{k,0}(f) \lrcorner \debar \tau^{n-1}\\
&=& (-1)^{k+1}\int_X \debar T_{k,0}(f) \lrcorner \tau^{n-1} = (-1)^{2k}\int_X \delta T_{k,1}(f) \lrcorner \tau^{n-1} \\
&=&(-1)^{2k}\int_X T_{k,1}(f) \lrcorner \delta^*\tau^{n-1} = \cdots = (-1)^{k(k+1)}\int_X T_{k,k}(f) \lrcorner \delta^*\tau^{n-k} \\
&=& \int_X \textrm{Dol}([f]) \wedge \textrm{Dol}^*([g]).
\end{eqnarray*}
\end{proof}

\section{Compatibility with the cup product}\label{sec:cup}
Assume that $X$ is compact and Cohen-Macaulay. In view of
\cite[Theorem~1.2]{AS} 
 and Theorem~\ref{thm:resol} we have that
\begin{equation}\label{eq:Doliso}
H^k(X, \hol_X)\cong H^k\left(\A^{0,\bullet}(X), \debar\right) \quad \textrm{and} \quad 
H^k(X, \omega_X^{n,0})\cong H^k\left(\B^{n,\bullet}(X), \debar\right),
\end{equation}
cf.\ the introduction. 
Now we make these Dolbeault isomorphisms explicit in a slightly different way than in the previous section:
We adopt in this section the standard definition of \v{C}ech cochain groups so that
now 
\begin{equation*}
C^p(\V, \F):=\prod_{\alpha_0\neq \alpha_1\neq \cdots \neq \alpha_p} \F(V_{\alpha_0}\cap\cdots \cap V_{\alpha_p})
\end{equation*}
for a sheaf $\F$ on $X$ and a locally finite open cover $\V=\{V_{\alpha}\}$.

Let $\V$ be a 
Leray covering and let $\{\chi_{\alpha}\}$ be a smooth partition of unity subordinate to $\V$.
Following \cite[Chapter IV, \S 6]{Demailly}, given \v{C}ech cocycles $c\in C^p(\V, \hol_X)$ and $c'\in C^q(\V, \omega_X^{n,0})$
we define \v{C}ech cochains $f\in C^0(\V, \A^{0,p}_X)$ and 
$f' \in C^0(\V, \B^{n,q}_X)$ by
\begin{equation*}
f_{\alpha}= \sum_{\nu_0,\ldots,\nu_{p-1}} \debar \chi_{\nu_0} \wedge \cdots \wedge \debar \chi_{\nu_{p-1}} \cdot
c_{\nu_0 \cdots \nu_{p-1} \alpha} \quad \textrm{in} \quad V_{\alpha},
\end{equation*}
\begin{equation*}
f'_{\alpha}= \sum_{\nu_0,\ldots,\nu_{q-1}} \debar \chi_{\nu_0} \wedge \cdots \wedge \debar \chi_{\nu_{q-1}} \wedge
c'_{\nu_0 \cdots \nu_{q-1} \alpha} \quad \textrm{in} \quad V_{\alpha}.
\end{equation*}
In fact, $f$ and $f'$ are cocycles and define $\debar$-closed global sections
\begin{equation}\label{eq:Dolmap1}
\varphi=\sum_{\nu_p} \chi_{\nu_p} f_{\nu_p}=
\sum_{\nu_0,\ldots,\nu_{p}} \chi_{\nu_p}\debar \chi_{\nu_0} \wedge \cdots \wedge \debar \chi_{\nu_{p-1}} \cdot
c_{\nu_0 \cdots \nu_{p}} \in \A^{0,p}(X),
\end{equation} 
\begin{equation}\label{eq:Dolmap2}
\varphi'=\sum_{\nu_q} \chi_{\nu_q} f'_{\nu_q}=
\sum_{\nu_0,\ldots,\nu_{q}} \chi_{\nu_q}\debar \chi_{\nu_0} \wedge \cdots \wedge \debar \chi_{\nu_{q-1}} \wedge
c'_{\nu_0 \cdots \nu_{q}} \in \B^{n,q}(X).
\end{equation}
The Dolbeault isomorphisms \eqref{eq:Doliso} are then realized by
\begin{equation*}
H^p(X, \hol_X) \stackrel{\simeq}{\longrightarrow} H^p(\A^{0,\bullet}(X)), \quad
[c] \mapsto [\varphi], \quad \textrm{and}
\end{equation*}
\begin{equation*}
H^q(X, \omega_X^{n,0}) \stackrel{\simeq}{\longrightarrow} H^q(\B^{n,\bullet}(X)), \quad
[c'] \mapsto [\varphi'],
\end{equation*}
respectively.

We can now show that the cup product is compatible with our trace map on the level of cohomology.

\begin{proposition}\label{prop:commutative}
The following diagram commutes.
\begin{equation*}
\begin{array}{ccc}
H^p(X,\hol_X) \times H^q(X, \omega_X^{n,0}) & \stackrel{\cup}{\longrightarrow} & H^{p+q}(X, \omega_X^{n,0}) \\
\downarrow & & \downarrow \\
H^p(\A^{0,\bullet}(X)) \times H^q(\B^{n,\bullet}(X)) & \stackrel{\wedge}{\longrightarrow} & H^{p+q}(\B^{n,\bullet}(X)),
\end{array}
\end{equation*}
where the vertical mappings are the Dolbeault isomorphisms.
\end{proposition}

\begin{proof}
Let $\V=\{V_{\alpha}\}$ be a Leray covering of $X$.
Let $[c]\in H^p(X, \hol_X)$ and $[c']\in H^q(X, \omega_X^{n,0})$, where 
$c\in C^p(\V, \hol_X)$ and $c'\in C^q(\V, \omega_X^{n,0})$ are cocycles. Then
$c\cup c' \in C^{p+q}(\V, \omega_X^{n,0})$, defined by
\begin{equation*}
(c\cup c')_{\alpha_0 \cdots \alpha_{p+q}} = c_{\alpha_0 \cdots \alpha_p} \cdot c'_{\alpha_p \cdots \alpha_{p+q}} \quad
\textrm{in} \,\, V_{\alpha_0}\cap \cdots \cap V_{\alpha_{p+q}},
\end{equation*}
is a cocycle representing $[c] \cup [c'] \in \check{H}^{p+q}(X, \omega_X^{n,0})$. The image of 
$[c]\cup [c']$ in $H^{p+q}(\B^{n,\bullet}(X))$ is the cohomology class defined by the $\debar$-closed current
\begin{equation}\label{eq:cuphack}
\sum_{\nu_0,\ldots,\nu_{p+q}} \chi_{\nu_{p+q}}\debar \chi_{\nu_0} \wedge \cdots \wedge \debar \chi_{\nu_{p+q-1}} \wedge
c_{\nu_0 \cdots \nu_{p}} \cdot c'_{\nu_p \cdots \nu_{p+q}} \in \B^{n,p+q}(X).
\end{equation} 

The images of $[c]$ and $[c']$ in Dolbeault cohomology are, respectively, the cohomology classes of the $\debar$-closed currents
$\varphi$ and $\varphi'$ defined by \eqref{eq:Dolmap1} and \eqref{eq:Dolmap2}. Notice that 
\begin{equation*}
\varphi|_{V_{\nu_p}} = 
\sum_{\nu_0,\ldots,\nu_{p-1}} \debar \chi_{\nu_0} \wedge \cdots \wedge \debar \chi_{\nu_{p-1}} \cdot
c_{\nu_0 \cdots \nu_{p-1} \nu_p}.
\end{equation*}
Therefore, $\varphi\wedge \varphi'$ is given by \eqref{eq:cuphack} as well.
\end{proof}

Notice that $H^n(X,\omega_X^{n,0})\simeq \C$ (e.g.\ as it is the dual of $H^0(X,\hol_X)$) and any two realizations of this 
isomorphism are the same up to a multiplicative constant. In the compact Cohen-Macaulay case it thus follows 
from Proposition~\ref{prop:commutative}
that the duality of this paper, up to a multiplicative constant, 
is the same as the abstractly defined duality in complex and algebraic geometry.

\end{document}